\documentclass[12pt]{amsart}
\usepackage{amsmath}
\usepackage{amsthm}
\usepackage{mathrsfs,amssymb,amsfonts,extarrows,textcomp,bbold,leftidx}
\usepackage[all]{xy}
\numberwithin{equation}{section}

\newtheorem{theorem}[subsection]{Theorem}
\newtheorem*{theorem*}{Theorem}
\newtheorem{lemma}[subsection]{Lemma}

\newtheorem{corollary}[subsection]{Corollary}

\newtheorem{proposition}[subsection]{Proposition}

\theoremstyle{definition}
\newtheorem{example}[subsection]{Example}
\newtheorem{definition}[subsection]{Definition}
\newtheorem*{definition*}{Definition}
\newtheorem{notation}[subsection]{Notation}

\theoremstyle{remark}
\newtheorem{remark}[subsection]{Remark}
\newtheorem{construction}[subsection]{Construction}

\makeatletter\@addtoreset{equation}{section}

\makeatother

\usepackage{hyperref}

\newcommand{\KK}{\mathbb{K}}
\newcommand{\LL}{\mathbb{L}}

\newcommand{\PP}{\mathbb{P}}
\newcommand{\ZZ}{\mathbb{Z}}
\newcommand{\QQ}{\mathbb{Q}}

\newcommand{\cP}{\mathcal{P}}
\newcommand{\cG}{\mathcal{G}}

\newcommand{\Spec}{\operatorname{Spec}}
\newcommand{\Aut}{\operatorname{Aut}}

\newcommand{\rk}{\operatorname{rk}}

\newcommand{\Pic}{\operatorname{Pic}}
\newcommand{\rkPic}{\operatorname{rkPic}}

\newcommand{\Gal}{\operatorname{Gal}}

\newcommand{\WD}{\mathrm{W}(\mathrm{D}_5)}

\newcommand{\SSS}{\mathfrak{S}}

\newcommand{\NNN}{\mathfrak{N}}

\newcommand{\HHH}{\mathfrak{H}}
\newcommand{\et}{\acute{e}t}

\def \ge {\geqslant}
\def \le {\leqslant}

\title{Birational geometry of del Pezzo surfaces of degree~$4$}

\author{Constantin Shramov}
\author{Andrey Trepalin}

\address{\emph{Constantin Shramov}
\newline
\textnormal{Steklov Mathematical Institute of RAS,
8 Gubkina street, Moscow 119991, Russia.}
\newline
\textnormal{HSE University, Russian Federation,
Laboratory of Algebraic Geometry, 6 Usacheva str., Moscow, 119048, Russia.}
\newline
\textnormal{\texttt{costya.shramov@gmail.com}}}

\address{\emph{Andrey Trepalin}
\newline
\textnormal{Steklov Mathematical Institute of RAS,
8 Gubkina street, Moscow 119991, Russia.}
\newline
\textnormal{HSE University, Russian Federation,
Laboratory of Algebraic Geometry, 6 Usacheva str., Moscow, 119048, Russia.}
\newline
\textnormal{\texttt{trepalin@mccme.ru}}}

\begin{document}

\begin{abstract}
It is known that any Mori fiber space birational to a minimal
smooth del Pezzo surface $S$ of degree~$4$ is either a del Pezzo
surface of degree~$4$ itself, or a smooth cubic surface with a structure of a relatively minimal
conic bundle. We show that any del Pezzo surface of degree~$4$ birational to~$S$
is actually isomorphic to~$S$. Also, we sketch an equivariant
version of this fact.
On the way, we review the biregular classification of
del Pezzo surfaces of degree~$4$ obtained by A.\,N.\,Skorobogatov.
\end{abstract}

\maketitle
\tableofcontents

\section{Introduction}

A smooth geometrically irreducible projective surface $Z$
together with a morphism~\mbox{$\phi\colon Z\to B$}
is called a \emph{Mori fiber space} if $\dim Z>\dim B$, the anticanonical divisor of $Z$ is $\phi$-ample, one has
$\rkPic(Z)-\rkPic(B)=1$, and $\phi_*\mathcal{O}_Z=\mathcal{O}_B$. Thus, a Mori fiber space is either
a del Pezzo surface (i.e. a smooth geometrically irreducible projective surface with ample anticanonical divisor) of Picard rank~$1$,
or a relatively minimal conic bundle (i.e. a conic bundle of relative Picard rank~$1$).

Mori fiber spaces are end results of the Minimal Model Program ran on geometrically
ruled surfaces (see e.g.~\mbox{\cite[Theorem~2.7]{Mori}}),
and so they naturally arise from studying rationality questions.
It is known that if $S$ is a del Pezzo surface of degree $K_S^2\ge 5$ over
a field $\KK$ with $S(\KK)\neq\varnothing$, then $S$ is rational, see for instance~\mbox{\cite[Theorem~2.1]{VA}} or~\mbox{\cite[Theorem~1.2]{BT}}.
On the other hand,
if $\KK$ is perfect,
the degree of $S$ is at most~$3$, and $\rkPic(S)=1$, then the only Mori fiber space birational to $S$ is $S$ itself, see~\mbox{\cite[Theorem~1.6(ii) and~\S4]{Isk96}}; in particular $S$ is not rational in this case. The same assertion holds over non-perfect fields as well: one can easily deduce this from the results of~\cite{BFSZ} (see Theorem~\ref{theorem:lowdegree} below).
If $S$ is a del Pezzo surface of degree~$9$ over a field $\KK$, such that either $\KK$ is perfect, or its characteristic is different from~$2$ and~$3$, with $S(\KK)=\varnothing$, then $S$ is not birational to conic bundles, and there exist only two del Pezzo surfaces of Picard rank~$1$ birational to $S$, including $S$ itself; see \cite{Weinstein} or~\mbox{\cite[Corollary 2.4 and Theorem 2.10]{Sh20}}. If $S$ is a del Pezzo surface of degree $8$ of Picard rank $1$ over
a perfect field $\KK$ with $S(\KK)=\varnothing$, then~$S$ is birational to certain conic bundles, but the only del Pezzo surface of Picard rank~$1$ birational to $S$ is $S$ itself;
see~\mbox{\cite[Theorem 1.6]{Trepalin-dP8}}. Birational geometry of del Pezzo surfaces of degree~$6$ over perfect fields was studied in detail in~\cite{KY}.

Concerning del Pezzo surfaces of degree~$4$, the following results are known.
They were proved by V.\,A.\,Iskovskikh over a perfect field, and the case of an arbitrary field
follow from~\cite{BFSZ};
we refer the reader to Theorem~\ref{theorem:Iskovskikh-dP4-models} below for a more detailed statement and a proof.

\begin{theorem}
\label{theorem:Iskovskikh-old}
Let $S$ be a del Pezzo surface of degree $4$ over
a field $\KK$ such that~\mbox{$\rkPic(S)=1$}. The following assertions hold.
\begin{itemize}
\item[(i)] If $S(\KK)=\varnothing$, then the only Mori fiber space birational to $S$ is $S$ itself.

\item[(ii)] If $S(\KK)\neq \varnothing$, then all the Mori fiber spaces birational to $S$ are either del Pezzo
surfaces of Picard rank~$1$ and degree $4$,
or smooth cubic surfaces with a structure of a relatively minimal conic bundle over $\PP^1$.
In particular, $S$ is not rational.
\end{itemize}
\end{theorem}

The first goal of this paper is to prove the following assertion
which makes Theorem~\ref{theorem:Iskovskikh-old} more precise.

\begin{theorem}\label{theorem:dP4-unique}
Let $S$ be a del Pezzo surface of degree $4$ such that~{$\rkPic(S)=1$}.
Then the only del Pezzo surface of degree $4$ birational to $S$ is $S$ itself.
\end{theorem}

In particular, Theorem~\ref{theorem:dP4-unique} applied together
with Theorem~\ref{theorem:Iskovskikh-old}
implies that the only del Pezzo surface of Picard rank~$1$
birational to a del Pezzo surface~$S$ of degree~$4$ and Picard rank~$1$ is~$S$ itself.

It is interesting to compare Theorems~\ref{theorem:Iskovskikh-old} and~\ref{theorem:dP4-unique} with a result of the same
flavor which holds in a totally different setting, where Mori fiber spaces (and in particular del Pezzo surfaces) are not assumed
to be smooth, but only regular.

\begin{theorem}[{\cite[Theorem 4.38(2)]{BFSZ}}]
Let $S$ be a geometrically integral and geometrically non-normal
regular del Pezzo surface of degree $4$ over a field $\KK$ of characteristic~$2$
such that $\rkPic(S)=1$.
Then all the regular Mori fiber spaces birational to~$S$ are~$S$ itself, and
certain regular cubic surfaces with a structure of a relatively minimal conic bundle over~$\PP^1$.
\end{theorem}

We also give a sketch of a proof of a more general version of
Theorems~\ref{theorem:Iskovskikh-old} and~\ref{theorem:dP4-unique} taking into account an action of a finite group.

\begin{theorem}\label{theorem:dP4-unique-equivariant}
Let $S$ be a del Pezzo surface of degree $4$ over
a field $\KK$
with an action of a finite group $\Gamma$ such that $\rkPic(S)^{\Gamma}=1$.
The following assertions hold.
\begin{itemize}
\item[(i)] If there are no $\Gamma$-invariant $\KK$-points on $S$, then the only $\Gamma$-Mori fiber space \mbox{$\Gamma$-equivariantly} birational to $S$ is $S$ itself.
In particular, $S$ is not $\Gamma$-rational.

\item[(ii)] If there exists
a $\Gamma$-invariant $\KK$-point on $S$, then all the $\Gamma$-Mori fiber spaces \mbox{$\Gamma$-equivariantly} birational to $S$ are
$S$ itself and smooth cubic surfaces with a~structure of a relatively $\Gamma$-minimal conic bundle over $\PP^1$
obtained by a blow up of \mbox{a $\Gamma$-invariant} $\KK$-point on~$S$.
In particular, $S$ is not $\Gamma$-rational.
\end{itemize}
\end{theorem}

In course of the proofs of our main results, we use the following
theorem due to A.\,N.\,Skorobogatov.

\begin{theorem}[{see \cite[Theorem~2.3]{Skorobogatov-Kummer}}]
\label{theorem:Skorobogatov-short}
There is a natural one-to-one correspondence
between isomorphism classes of del Pezzo surfaces of degree~$4$ over a field $\KK$,
and the set of pairs~\mbox{$(\Delta,\Lambda)$},
where $\Delta\subset\PP^1$ is a smooth closed subscheme of length~$5$, and $\Lambda$ is an isomorphism class of torsors
over the group scheme~\mbox{$(R_{\KK[\Delta]/\KK}\,(\ZZ/2\ZZ))/(\ZZ/2\ZZ)$};
here~$R_{\KK[\Delta]/\KK}$ denotes the Weil restriction of scalars.
\end{theorem}

A more precise version of Theorem~\ref{theorem:Skorobogatov-short} is Theorem~\ref{theorem:Skorobogatov}.
In~\cite{Skorobogatov-Kummer}, this theorem was proved under certain minor assumptions:
the characteristic of~$\KK$ was assumed to be different from~$2$, and $\KK$ was assumed to contain
at least~$5$ elements. We take an opportunity to recall the beautiful proof
from~\cite{Skorobogatov-Kummer} with some minor additional work which
allows to get rid of these assumptions on~$\KK$ (cf. Remark~\ref{remark:Skorobogatov-upgrade}).

\begin{remark}
In a recent paper~\cite{ESS}, a totally different approach based on the results of~\cite{Elagin}
was used to prove Theorems~\ref{theorem:dP4-unique} and~\ref{theorem:dP4-unique-equivariant}
over perfect fields.
\end{remark}

\medskip
Some of our results can be generalized to a wider case of regular (possibly non-smooth) del Pezzo surfaces.
Recall that a regular del Pezzo surface $S$ with $\rkPic(S)=1$
is called \emph{birationally rigid}, if any regular Mori fiber space birational to $S$ is isomorphic to~$S$.
Furthermore, $S$ is called \emph{birationally super-rigid}, if any birational map $S \dashrightarrow S'$ to a regular Mori fiber space $S'$ is an isomorphism.
The following assertion is a generalization of Theorem~\ref{theorem:Iskovskikh-old}(i).

\begin{theorem}\label{theorem:pointless-dP4}
Let $S$ be a geometrically integral regular del Pezzo surface of degree~$4$ over
a field~$\KK$ such that~\mbox{$\rkPic(S)=1$} and
$S(\KK)=\varnothing$. Then $S$ is birationally rigid. Moreover, if~$S$ has no points of degree $2$ and $3$, then it is
birationally super-rigid.
\end{theorem}

Although this lies apart from the main goals of the paper, we
prove the following theorem which is well known over perfect fields, and goes back to Yu.\,I.\,Manin, see~\mbox{\cite[Theorem~33.2]{Man74}}.

\begin{theorem}
\label{theorem:lowdegree}
Let $S$ be a geometrically integral regular del Pezzo surface of degree $d \leqslant 3$ over a field $\KK$ such that~\mbox{$\rkPic(S)=1$}.
Then $S$ is birationally rigid. Moreover, if $S(\KK)=\varnothing$ or $d = 1$, then $S$ is birationally super-rigid.
\end{theorem}

Suppose that $S$ is a regular
del Pezzo surface with an action of a finite group $\Gamma$ such that $\rkPic(S)^\Gamma=1$.
Then $S$ is called \emph{$\Gamma$-birationally rigid}, if any regular $\Gamma$-Mori fiber space $\Gamma$-birational to $S$ is $\Gamma$-isomorphic to $S$.
Furthermore, $S$ is called \emph{$\Gamma$-birationally super-rigid}, if any $\Gamma$-birational map $S \dashrightarrow S'$ to a regular $\Gamma$-Mori fiber space $S'$ is an isomorphism.
The next assertion is a generalization of Theorem~\ref{theorem:dP4-unique-equivariant}(i).

\begin{theorem}\label{theorem:G-pointless-dP4}
Let $S$ be a geometrically integral regular del Pezzo surface of degree $4$ over
a field $\KK$ with an action of a finite group $\Gamma$ such that $\rkPic(S)^{\Gamma}=1$
and there are no $\Gamma$-invariant $\KK$-points on $S$.
Then $S$ is \mbox{$\Gamma$-birationally} rigid. Moreover, if $S$ has no $\Gamma$-invariant points of degree~$2$ and~$3$,
and no $\Gamma$-invariant pairs and triples of $\KK$-points,
then it is birationally super-rigid.
\end{theorem}

We also establish an equivariant version of Theorem~\ref{theorem:lowdegree}.

\begin{theorem}
\label{theorem:G-lowdegree}
Let $S$ be a geometrically integral regular del Pezzo surface of degree~\mbox{$d \leqslant 3$} over a field $\KK$ with an action of a finite group $\Gamma$ such that~\mbox{$\rkPic(S)^{\Gamma}=1$}.
Then~$S$ is \mbox{$\Gamma$-birationally} rigid. Moreover, if there are no $\Gamma$-invariant $\KK$-points on $S$, or if $d = 1$, then~$S$ is $\Gamma$-birationally super-rigid.
\end{theorem}

\begin{remark}
It may be interesting to find out whether Theorems~\ref{theorem:dP4-unique} and~\ref{theorem:dP4-unique-equivariant}(ii)
can be generalized to the case of regular del Pezzo surfaces.
\end{remark}

\medskip
The plan of the paper is as follows.
In Section~\ref{section:markings} we recall the basic properties of Weyl groups~$\mathrm{W}(\mathrm{D}_k)$
and the notion of a marking which allows to fix an action of such a group on irreducible components of degenerate fibers of a conic bundle,
or (for the group~$\WD$) on the lines on a del Pezzo surface of degree~$4$.
In Section~\ref{section:blow-up} we discuss the way to construct
a marking of a conic bundle obtained as a blow up of a del Pezzo surface $S$ of degree~$4$ at a point
from a marking of~$S$, and vice versa.
In Section~\ref{section:pencil}
we recall the properties of quadrics passing through a del
Pezzo surface $S\subset\PP^4$ of degree~$4$.
In Section~\ref{section:biregular}
we discuss biregular classification of del Pezzo surfaces of degree~$4$ and prove Theorem~\ref{theorem:Skorobogatov}
(which implies Theorem~\ref{theorem:Skorobogatov-short}) following~\cite{Skorobogatov-Kummer}.
In Section~\ref{section:birational} we study birational models of minimal del Pezzo surfaces of degree~$4$
and prove Theorem~\ref{theorem:dP4-unique}.
In Section~\ref{section:G-birational}
we provide a~sketch of a proof of Theorem~\ref{theorem:dP4-unique-equivariant}
which follows the same lines as the proof of Theorem~\ref{theorem:dP4-unique}.
In Appendix~\ref{appendix:involutions}
we collect some facts about Geiser and Bertini involutions, prove Theorem~\ref{theorem:lowdegree}, and
sketch a proof of Theorem~\ref{theorem:G-lowdegree}.
In Appendix~\ref{appendix:pointless-4}, we recall the results of~\cite{BFSZ}
concerning Sarkisov links from geometrically integral regular del Pezzo surfaces of degree~$4$,
prove Theorem~\ref{theorem:pointless-dP4}, and sketch a proof of Theorem~\ref{theorem:G-pointless-dP4}.

While we provide complete proofs for all our theorems which do not involve group actions on varieties,
we give only sketches of proofs (which mostly use the same arguments as in the non-equivariant case)
for similar equivariant assertions, i.e. for Theorems~\ref{theorem:dP4-unique-equivariant},~\ref{theorem:G-pointless-dP4}, and~\ref{theorem:G-lowdegree},
as well as for some auxiliary results.
To turn the sketches into rigorous proofs, one needs to spell out the basics of $\Gamma$-Sarkisov theory for regular surfaces
over an arbitrary field with an action of a finite group $\Gamma$, similarly to what is done in~\cite{BFSZ}
for surfaces without group action. Namely, we need the existence of decomposition of a $\Gamma$-equivariant birational map into $\Gamma$-Sarkisov links,
and a classification of such links.

Throughout the paper we assume that all varieties are projective.
Given a field $\KK$, we denote by $\bar{\KK}$ its algebraic closure,
and by~$\KK^{sep}$ its separable closure.
If $X$ is a variety (or a scheme) defined over $\KK$, and $\LL$ is a $\KK$-algebra, we denote by $X_{\LL}$
the scheme~\mbox{$X\otimes_{\Spec\KK}\Spec\LL$} over~$\LL$; similarly, if $\phi$ is a morphism of varieties (or schemes) over $\KK$,
by $\phi_\LL$ we denote the corresponding morphism over~$\LL$.
By a point of degree $d$ on a variety defined over some field~$\KK$ we mean a closed point whose
residue field is an extension of~$\KK$ of degree~$d$; a $\KK$-point is a point of degree~$1$.
If $Y$ is a scheme over $\LL$, by $R_{\LL/\KK}\,Y$ we denote the Weil restriction of scalars
(see e.g.~\cite{Weil-restriction} for definitions and basic facts concerning this construction).
By $\SSS_n$ we denote the symmetric group on~$n$ letters.

We emphasize once again that in this paper we work only with \emph{smooth} del Pezzo surfaces and conic bundles, unless otherwise explicitly stated;
the more general case of possibly non-smooth regular del Pezzo surfaces will appear only in Appendices~\ref{appendix:involutions}
and~\ref{appendix:pointless-4}.
However, the proofs of our main theorems rely on the results of~\cite{BFSZ} concerning regular del Pezzo surfaces and conic bundles, which may a priori appear in Sarkisov links
connecting two smooth surfaces.

\medskip
We are grateful to Fabio Bernasconi, Alexander Kuznetsov, and Yuri Prokhorov for useful discussions.
This work was performed at the Steklov International Mathematical Center and supported by the Ministry of Science and Higher Education of the Russian Federation (agreement no.~075-15-2022-265) and by
the HSE University Basic Research Program.

\section{Markings}
\label{section:markings}

Let us start by recalling some properties of the Weyl group $\mathrm{W}(\mathrm{D}_k)$, $k\ge 4$.
We refer the reader to~\mbox{\cite[\S\,VI.4.8]{Bou02}} for more details.

\begin{notation}\label{notation:Weyl-group}
The Weyl group $\mathrm{W}(\mathrm{D}_k)$ is isomorphic to $(\ZZ/2\ZZ)^{k-1} \rtimes \SSS_k$. The normal subgroup~$(\ZZ/2\ZZ)^{k-1}$ consists of elements $\iota_{i_1 \ldots i_l}$, where $\{i_1, \ldots, i_l\}$ is a subset of $\{1, \ldots, k\}$ and $l$ is even number. One has
$$
\iota_{i_1 \ldots i_l} \cdot \iota_{j_1 \ldots j_m} = \iota_{s_1 \ldots s_r},
$$
where
$$
\{s_1, \ldots, s_r\} = \{i_1, \ldots, i_l\} \triangle \{j_1, \ldots, j_m\}.
$$
For an element $\sigma$ in the non-normal subgroup $\HHH\cong\SSS_k$ in $\mathrm{W}(\mathrm{D}_k)$ one has
$$
\sigma \iota_{i_1 \ldots i_l} \sigma^{-1} = \iota_{\sigma(i_1) \ldots \sigma(i_l)}.
$$
Also note that for $k = 5$ the normal subgroup $\NNN\cong(\ZZ/2\ZZ)^4$ is generated by the elements $\iota_{klmn}$ with the only relation
$$
\iota_{1234} \iota_{1235} \iota_{1245} \iota_{1345} \iota_{2345} = 1.
$$
\end{notation}

From now on we fix Notation~\ref{notation:Weyl-group}
for the elements of $\mathrm{W}(\mathrm{D}_k)$ to compare the action of Weyl groups on different objects.
In particular, we fix a non-normal subgroup~\mbox{$\HHH\subset\mathrm{W}(\mathrm{D}_k)$} isomorphic to~$\SSS_k$.

We are going to work with conic bundles and del Pezzo surfaces over arbitrary fields, and consider the actions of various
groups on the configurations of $(-1)$-curves on such surfaces. Recall that a $(-1)$-curve on a smooth projective surface is a
rational curve with self-intersection equal to~$-1$.
We will use without explicitly reference the following result which is well known to experts.

\begin{theorem}
Let $S$ be a smooth projective geometrically rational surface over
a field~$\KK$. The following assertions hold.
\begin{itemize}
\item[(i)] The natural homomorphism
$$
\Pic(S_{\KK^{sep}})\to \Pic(S_{\bar{\KK}})
$$
is an isomorphism.

\item[(ii)] If $S$ is a del Pezzo surface, then every $(-1)$-curve on $S_{\bar{\KK}}$ is defined over $\KK^{sep}$.

\item[(iii)] If $S$ has a structure of a conic bundle $\phi\colon S\to C$, then every irreducible component of every degenerate fiber of $\phi_{\bar{\KK}}$ is defined over $\KK^{sep}$.

\item[(iv)] If $S$ is a del Pezzo surface of degree $d\neq 8$, then there exists a birational contraction
$$
S_{\KK^{sep}}\to \PP^2_{\KK^{sep}}
$$
which blows down $9-d$ exceptional curves.

\item[(v)] If $S$ is a del Pezzo surface of degree $d\le 5$, then the natural homomorphism
$$
\Aut(S_{\KK^{sep}})\to \Aut(S_{\bar{\KK}})
$$
is an isomorphism.
\end{itemize}
\end{theorem}

\begin{proof}
For assertion~(i), we refer the reader to \cite[Theorem~4.2]{Liedtke} (see also~\cite{Coombes} for a more basic and general result).

Observe that an irreducible component of a degenerate fiber of a conic bundle is a $(-1)$-curve.
Now, if $S$ is a smooth projective geometrically rational surface, then every class in $\Pic(S_{\bar{\KK}})$ is defined over $\KK^{sep}$ by assertion~(i).
In particular, if $E$ is a $(-1)$-curve on~$S_{\bar{\KK}}$, then its class $[E]$ is defined over $\KK^{sep}$.
On the other hand, the space of global sections of the corresponding line bundle has dimension~$1$,
which implies that there exists a unique effective divisor $E'$ over $\KK^{sep}$ whose class in $\Pic(S_{\KK^{sep}})$ is $[E]$.
We see that $E'$ is a geometrically rational curve defined over $\KK^{sep}$, which implies that it is actually rational.
Hence, $E'$ is a $(-1)$-curve. This proves assertions~(ii) and~(iii).

Assertion~(iv) follows from assertion~(ii) and the existence of the required contraction over~$\bar{\KK}$.
For assertion~(v), see for instance~\mbox{\cite[Theorem~4(i)]{Vikulova}}.
\end{proof}

For a conic bundle $\phi\colon X\to\PP^1$ over a field $\KK$ with $k\ge 4$
geometric degenerate fibers,
the group of automorphisms of~$\Pic(X_{\KK^{sep}})$ preserving
the anticanonical class, the class of a fiber, and the intersection form is the Weyl group~$\mathrm{W}(\mathrm{D}_k)$. In particular, the group~$\mathrm{W}(\mathrm{D}_k)$ faithfully acts on the set of irreducible components of degenerate fibers of the conic bundle~\mbox{$\phi_{\KK^{sep}}\colon X_{\KK^{sep}}\to\PP^1$} (see \cite[Proposition 0.4e]{KT}). Moreover, an action of~$\mathrm{W}(\mathrm{D}_k)$ on the set of irreducible components of degenerate fibers of $\phi_{\KK^{sep}}$ induces an action of~$\mathrm{W}(\mathrm{D}_k)$ on~$\Pic(X_{\KK^{sep}})$, since~$\Pic(X_{\KK^{sep}})\otimes\QQ$ is generated by $-K_{X_{\KK^{sep}}}$ and the classes of irreducible components of degenerate fibers, and $\Pic(X_{\KK^{sep}})$ has no torsion.

For a del Pezzo surface $S$ of degree $4$ over a field $\KK$, the group of automorphisms of~$\Pic(S_{\KK^{sep}})$ preserving
the anticanonical class and the intersection form is the Weyl group~$\WD$.
In particular, the group $\WD$ acts on the set of lines on~$S_{\KK^{sep}}$;
we recall that $S_{\KK^{sep}}$ contains $16$ lines, no three of which pass through a single point,
and the configuration of lines on $S_{\KK^{sep}}$ (i.e., the incidence relation between the lines)
does not depend on~$S$. Moreover, an action of~$\WD$ on the set of lines on~$S_{\KK^{sep}}$
induces an action of~$\WD$ on~$\Pic(S_{\KK^{sep}})$, since~$\Pic(S_{\KK^{sep}})$ is generated by the classes of lines.
The group $\WD$ is the complete automorphism group of the configuration of lines on~$S_{\KK^{sep}}$ (see~\mbox{\cite[26.9]{Man74}}).

\begin{remark}
Note that an action of the Weyl group $\WD$ on the configuration of lines on~$S_{\KK^{sep}}$ is not unique.
For example, for a given action of $\WD$ any automorphism of~$\WD$ gives another action.
But if we fix an action of $\WD$ on the configuration of lines on~$S_{\KK^{sep}}$, then this action provides homomorphisms~\mbox{$\Gal(\KK^{sep}/\KK)\to \WD$}
and~\mbox{$\Aut(S)\to\WD$}. The same applies to actions of the
Weyl group $\mathrm{W}(\mathrm{D}_k)$ on irreducible components of geometric degenerate fibers
of a conic bundle.
\end{remark}

The following properties of the above action
are well known.

\begin{lemma}\label{lemma:Weyl-group-on-dP4-basic}
Let $S$ be a del Pezzo surface of degree $4$ over a field $\KK$, let $\Lambda$ denote the set of lines on $S_{\KK^{sep}}$,
so that the group $\WD$ acts on $\Lambda$ as the complete automorphism group of the configuration of lines on~$S_{\KK^{sep}}$.
Then
\begin{itemize}
\item[(i)] the action of the Weyl group $\WD$ on $\Lambda$ is faithful;

\item[(ii)] the action of the normal subgroup $\NNN\cong (\ZZ/2\ZZ)^4$ of $\WD$ on $\Lambda$ is free and transitive;

\item[(iii)] the homomorphism $\Aut(S)\to\WD$ is injective.
\end{itemize}
\end{lemma}

\begin{proof}
A del Pezzo surface $S_{\KK^{sep}}$ of degree $4$ is isomorphic
to the blowup of $\pi\colon S_{\KK^{sep}} \rightarrow \PP^2_{\KK^{sep}}$
at~$5$ points $p_1, \ldots, p_5$ in general position. Put $L_i = \pi^{-1}(p_i)$,  and $L = \pi^*l$, where~$l$ is the class of a line on $\PP^2_{\KK^{sep}}$.
The $16$ lines on $S_{\KK^{sep}}$ are $L_i$, the proper transforms of $10$ lines passing through two points $p_i$ and $p_j$, and the proper transform $Q$ of the
conic passing through all the points $p_i$. Note that
$$
Q\sim 2L - \sum\limits_{i = 1}^5 L_i.
$$

Assume that $g \in \WD$ acts trivially on $\Lambda$. Then $g$ preserves the classes of $L_i$ and $Q$. Therefore $g$ also preserves the class $L$, and trivially acts on~$\Pic(S_{\KK^{sep}})$. It is possible only for trivial $g$ since $\WD$ faithfully acts on~$\Pic(S_{\KK^{sep}})$. Thus $\WD$ faithfully acts on $\Lambda$.
This proves assertion~(i).

The stabilizer of $Q$ in $\WD$ contains a group $H \cong \SSS_5$ permuting the lines $L_i$: for~\mbox{$\sigma \in H$} one has $\sigma L_i = L_{\sigma(i)}$. The group $\NNN \cap H$ is a normal subgroup of order~\mbox{$2^r$, $r\leqslant 4$}, in $H$, therefore this group is trivial.
One has $|\WD| = |\NNN| \cdot |H|$. Thus~\mbox{$\WD \cong \NNN \rtimes H$}.
(We point out that~$H$ is not required to coincide with the subgroup~\mbox{$\HHH\cong\SSS_5$} fixed after Notation~\ref{notation:Weyl-group}.)

The group $\WD$ transitively acts on $\Lambda$ (see \cite[Corollary 26.7]{Man74}), therefore the stabilizer of $Q$ in $\NNN$ is trivial, and any line in $\Lambda$ can be uniquely represented as $\iota Q$, where~\mbox{$\iota \in \NNN$}. Thus the action of $\NNN$ on $\Lambda$ is transitive and free. This proves assertion~(ii).

Now assume that $g \in \Aut(S)$ acts trivially on $\Lambda$. Then the contraction $\pi\colon S_{\KK^{sep}} \rightarrow \PP^2_{\KK^{sep}}$ is equivariant with respect
to the cyclic group generated by $g$, and $g$ acts on $\PP^2_{\KK^{sep}}$ with five fixed points in general position. This is possible only if $g$ is trivial, which proves assertion~(iii).
\end{proof}

\begin{definition}[{see e.g. \cite{Hosoh}}]
Let $S$ be a del Pezzo surface of degree $4$ over a field~$\KK$. A \emph{marking} of $S$
is an ordered quintuple of lines on $S_{\KK^{sep}}$ intersecting some line
on~$S_{\KK^{sep}}$.
\end{definition}

One can see that the five lines in a marking of a del Pezzo surface of degree $4$ are pairwise
disjoint. Furthermore, these five lines are uniquely defined by the line they intersect, and vice versa
the latter line is uniquely defined by a quintuple of pairwise disjoint lines on~$S_{\KK^{sep}}$.
In other words, to define a marking of $S$
one has to choose a line~$Q$ on $S_{\KK^{sep}}$ and label the lines intersecting $Q$ as $L_1,\ldots,L_5$.
We will refer to such a marking as~\mbox{$L_1,\ldots,L_5$}, because
the line $Q$ will usually be not important for further constructions.

Note that a marking of $S$ allows to fix an action of $\WD$ on the set of lines on~$S_{\KK^{sep}}$. For any $\sigma \in \HHH$ (see Notation \ref{notation:Weyl-group}) we put $\sigma(Q) = Q$ and $\sigma (L_i) = L_{\sigma(i)}$, and for~\mbox{$\iota_{klmn} \in \NNN$} we put~\mbox{$\iota_{klmn}(Q) = L_i$}, where the set of indices $\{i, k, l, m, n\}$ coincides with
the set~\mbox{$\{1, 2, 3, 4, 5\}$}.

\begin{definition}[{see e.g. \cite[Definition~1.1]{Skorobogatov-Invariants}}]
\label{definition:marking-CB}
Let $\phi\colon X\to\PP^1$ be a conic bundle over a field $\KK$.
The \emph{marking} of $\phi$ is the ordering on the degenerate fibers of the conic bundle~\mbox{$\phi_{\KK^{sep}}\colon X_{\KK^{sep}}\to\PP^1$}
together with the choice of one of the two irreducible components in each degenerate fiber.
\end{definition}

In other words, to define a marking of $\phi$ as in Definition~\ref{definition:marking-CB},
one has to label the irreducible components of the degenerate fibers of $\phi_{\KK^{sep}}$ as
$E_1,\ldots,E_k, F_1,\ldots,F_k$ so that the degenerate fibers are $E_i\cup F_i$, $1\le i\le k$.
We will refer to such a marking as
$$
E_1\cup F_1,\ldots, E_k\cup F_k,
$$
assuming that the
chosen components of degenerate fibers are $E_1,\ldots,E_k$.

\begin{definition}
Let $\phi\colon X\to\PP^1$ be a conic bundle over a field $\KK$ such that
$\phi_{\KK^{sep}}$ has $k$ degenerate fibers.
We say that two markings $E_1\cup F_1,\ldots, E_k\cup F_k$ and $E_1'\cup F_1',\ldots, E_k'\cup F_k'$
of $\phi$ are \emph{equivalent} if one has $E_i=F_i'$ for all~\mbox{$1\le i\le k$}.
In other words, the equivalent markings of $X$ are $E_1\cup F_1,\ldots, E_k\cup F_k$
and $F_1\cup E_1,\ldots, F_k\cup E_k$.
\end{definition}

Note that a marking of $\phi$ allows to fix an action of $\mathrm{W}(\mathrm{D}_k)$ on the set of irreducible components of degenerate fibers of $\phi_{\KK^{sep}}$.
For any $\sigma \in \HHH$ (see Notation \ref{notation:Weyl-group}) we put~\mbox{$\sigma (E_i) = E_{\sigma(i)}$} and~\mbox{$\sigma (F_i) = F_{\sigma(i)}$}, and for $\iota_{i_1 \ldots i_l} \in \NNN$ we put $\iota_{i_1 \ldots i_l}(E_j) = F_j$ if~\mbox{$j \in \{i_1, \ldots, i_l\}$}, and~\mbox{$\iota_{i_1 \ldots i_l}(E_j) = E_j$} otherwise.

\begin{remark} \label{remark: equivalent-markings-gives-same-wd-actions}
It is easy to see that two equivalent markings of a conic bundle~$\phi$ give the same actions of $\mathrm{W}(\mathrm{D}_k)$
on the set of irreducible components of degenerate fibers of~$\phi_{\KK^{sep}}$.
\end{remark}

Now we are going to make several observations concerning sections and degenerate fibers
of conic bundles.

\begin{lemma}\label{lemma:intersection-of-sections}
Let $\phi\colon X\to\PP^1$ be a conic bundle over a field $\KK$, and let $E_1\cup F_1,\ldots, E_k\cup F_k$
be a marking of $\phi$. Let $A_1$ and $A_2$ be two sections of $\phi_{\KK^{sep}}$. Then
$$
2A_1 \cdot A_2 = A_1^2 + A_2^2 + l_1 + l_2 - 2r,
$$
where $l_i$ is the number of curves among $E_1,\ldots,E_k$ that intersect $A_i$, and $r$ is the number of curves among $E_1,\ldots,E_k$ that intersect both $A_1$ and $A_2$.
\end{lemma}

\begin{proof}
The Picard group of $X_{\KK^{sep}}$ is generated by the class of a section $C$ of $\phi_{\KK^{sep}}$ that does not intersect any of the curves $E_i$,
by the class of a fiber $F$, and by the classes of the curves $E_i$, see~\mbox{\cite[Proposition 0.4a]{KT}}. One has
$$
F^2 = F \cdot E_i = E_i \cdot E_j = C \cdot E_i = 0, \qquad C \cdot F = 1, \qquad C^2 = n.
$$
for $i \neq j$.

Assume that the section $A_1$ of $\phi_{\KK^{sep}}$ intersects $E_{i_1}, \ldots, E_{i_{l_1}}$ and does not intersect the other curves $E_i$, the section $A_2$ of $\phi_{\KK^{sep}}$ intersects $E_{j_1}, \ldots, E_{j_{l_2}}$ and does not intersect the other curves $E_j$. Then the sections $A_1$ and $A_2$ have classes $C + m_1F - E_{i_1} - \ldots - E_{i_{l_1}}$ and $C + m_2F - E_{j_1} - \ldots - E_{j_{l_2}}$, respectively. One has
$$
A_1^2 = C^2 + 2m_1C \cdot F + E_{i_1}^2 + \ldots + E_{i_{l_1}}^2 = n + 2m_1 - l_1.
$$
Similarly, $A_2^2 = n + 2m_2 - l_2$. Let $E_{s_1}$, $\ldots$, $E_{s_r}$ be the curves among $E_1,\ldots,E_k$ that intersect both $A_1$ and $A_2$. Then
\begin{multline*}
2A_1 \cdot A_2 = 2C^2 + 2m_1C \cdot F + 2m_2C \cdot F + 2E_{s_1}^2 + \ldots + 2E_{s_r}^2 = 2n + 2m_1 + 2m_2 - 2r =
\\
= (n + 2m_1 - l_1) + (n + 2m_2 - l_2) + l_1 + l_2 - 2r = A_1^2 + A_2^2 + l_1 + l_2 - 2r.
\end{multline*}
\end{proof}

\begin{corollary}[{cf. \cite[Proposition~4.1]{Skorobogatov-Invariants}}]
\label{corollary:parity-well-defined}
Let $\phi\colon X\to\PP^1$ be a conic bundle over a field $\KK$, and let $E_1\cup F_1,\ldots, E_k\cup F_k$
be a marking of $\phi$. Then one and only one of the following two assertions hold:
\begin{itemize}
\item every $(-1)$-curve on $X_{\KK^{sep}}$ that is a section of $\phi_{\KK^{sep}}$ intersects an even number of curves among $E_1,\ldots,E_k$;

\item every $(-1)$-curve on $X_{\KK^{sep}}$ that is a section of $\phi_{\KK^{sep}}$ intersects an odd number of curves among $E_1,\ldots,E_k$.
\end{itemize}
\end{corollary}

\begin{proof}
Let $A_1$ and $A_2$ be two $(-1)$-curves that are sections of $\phi_{\KK^{sep}}$. Then by Lemma~\ref{lemma:intersection-of-sections} one has
$$
2A_1 \cdot A_2 = -2 + l_1 + l_2 - 2r.
$$
Therefore $l_1$ and $l_2$ have the same parity.
\end{proof}

\begin{corollary}
\label{corollary:odd-number-of-fibers}
Let $\phi\colon X\to\PP^1$ be a conic bundle over a field $\KK$
such that the number~$k$ of degenerate fibers of $\phi_{\KK^{sep}}$ is odd. Let $E_1\cup F_1,\ldots, E_k\cup F_k$
be a marking of $\phi$. Then one and only one of the following two assertions hold:
\begin{itemize}
\item every $(-1)$-curve on $X_{\KK^{sep}}$ that is a section of $\phi_{\KK^{sep}}$ intersects an even number of curves among $E_1,\ldots,E_k$,
and every such $(-1)$-curve intersects an odd number of curves among $F_1,\ldots,F_k$;

\item every $(-1)$-curve on $X_{\KK^{sep}}$ that is a section of $\phi_{\KK^{sep}}$ intersects an odd number of curves among $E_1,\ldots,E_k$,
and every such $(-1)$-curve intersects an even number of curves among $F_1,\ldots,F_k$.
\end{itemize}
\end{corollary}

\begin{proof}
Let $A$ be a section of $\phi_{\KK^{sep}}$. If $A$ intersects the curve $E_i$, then it does not intersect~$F_i$, and vice versa.
Thus, $A$ intersects an even number of curves among $E_1,\ldots,E_k$ if and only if it intersects an odd number of curves among $F_1,\ldots,F_k$.
The rest follows from Corollary~\ref{corollary:parity-well-defined}.
\end{proof}

\begin{corollary}\label{corollary:markings-differ-sections}
Let $\phi\colon X\to\PP^1$ be a conic bundle, and let $E_1\cup F_1,\ldots, E_k\cup F_k$
be a marking of $\phi$. Let $A_1$ and $A_2$ be sections of $\phi_{\KK^{sep}}$ with negative self-intesection numbers. Then $A_1$ and $A_2$ cannot intersect
the same collection of curves among $E_1,\ldots,E_k$.
\end{corollary}

\begin{proof}
Suppose that both $A_1$ and $A_2$ intersect $E_{i_1}, \ldots, E_{i_{l}}$ and do not intersect the other curves $E_i$.
Then by Lemma~\ref{lemma:intersection-of-sections} one has
$$
2A_1 \cdot A_2 = A_1^2 + A_2^2 + l + l - 2l = A_1^2 + A_2^2 < 0,
$$
which is impossible since the sections are irreducible.
\end{proof}

\section{Blow up at a point}
\label{section:blow-up}

Recall that any smooth cubic surface $X$ over a field $\KK$ containing a line defined over~$\KK$
has a structure of a conic bundle $\phi\colon X\to\PP^1$ given by the linear projection from this line.
Conversely, any conic bundle structure on a smooth cubic surface is obtained by a projection from a line.

\begin{lemma}
\label{lemma:unique-bisection}
Let $X$ be a smooth cubic surface admitting a conic bundle
structure
$$
\phi\colon X\to\PP^1.
$$
Let $F$ be a general fiber of $\phi$.
Then there is a line $E\sim -K_X-F$. In other words, $\phi$ is given by the linear projection from $E$.
\end{lemma}

\begin{proof}
Then $-K_X \cdot F = 2$, therefore $F \subset X \subset \PP^3$ is a curve of degree $2$ in $\PP^3$ that is a conic. Consider a plane $\Pi$ containing $F$. Then $\Pi \cap X$ is a reducible cubic curve consisting of the conic $F$ and a line $E$ with class $-K_X-F$. Therefore $\phi$ is given by the linear projection from $E$.
\end{proof}

\begin{remark}
If in the notation of Lemma~\ref{lemma:unique-bisection} we additionally assume that~\mbox{$\rkPic(X)=2$}, then a line on~$X$
constructed in the lemma is unique,
and hence the conic bundle structure is unique as well.
\end{remark}

\begin{remark}
\label{remark:lines-on-cubic}
The $27$ lines on a smooth cubic surface
$X_{\KK^{sep}}$ over a separably closed field~$\KK^{sep}$ are $10$ irreducible components of degenerate fibers of a
conic bundle $\phi_{\KK^{sep}}\colon X_{\KK^{sep}}\to\PP^1$,
a unique (possibly inseparable) bi-section of $\phi_{\KK^{sep}}$ with self-intersection equal to $-1$, and~$16$ sections of $\phi_{\KK^{sep}}$
with self-intersection equal to $-1$. Moreover, this bi-section does not intersect any of the $16$ lines that are sections of~$\phi_{\KK^{sep}}$, and transversally
intersects each irreducible component of the degenerate fibers.
\end{remark}

Note that a blow up of a $\KK$-point on a del Pezzo surface $S$ of degree $4$ over a field $\KK$ is a smooth cubic surface if and
only if this point is not contained in any line on~$S_{\KK^{sep}}$.
Let~\mbox{$f\colon X \to S$} be the blowup of such a point.
Then there is a structure of a conic bundle~\mbox{$\phi\colon X\to\PP^1$} defined by the linear projection from the line that is the exceptional divisor of the blowup.
Therefore, the $16$ lines on~$S_{\KK^{sep}}$ are in one-to-one correspondence with sections of the conic bundle $\phi_{\KK^{sep}}\colon X_{\KK^{sep}}\to\PP^1$ having
self-intersection equal to~$-1$.

\begin{lemma}\label{lemma:unique-section-odd}
Let $X_{\KK^{sep}}$ be a smooth cubic surface over a separably closed field~$\KK^{sep}$, and let
$\phi_{\KK^{sep}}\colon X\to \PP^1$ be a conic bundle.
Let $E_1\cup F_1,\ldots,E_5\cup F_5$ be a marking of~$\phi_{\KK^{sep}}$.
Suppose that there exists a section of $\phi_{\KK^{sep}}$ with self-intersection~$-1$ that intersects an odd number of curves among~\mbox{$E_1,\ldots,E_5$}.
Then there exists a unique section of $\phi_{\KK^{sep}}$ with self-intersection~$-1$ that intersects
all the curves~\mbox{$E_1,\ldots,E_5$}.
Furthermore, for any~\mbox{$i\in\{1,\ldots,5\}$} there exists a unique section of $\phi_{\KK^{sep}}$ with self-intersection~$-1$ that intersects~$E_i$ but no other curve among~\mbox{$E_1,\ldots,E_5$}.
\end{lemma}

\begin{proof}
By Remark~\ref{remark:lines-on-cubic}, there are $16$ sections of $\phi_{\KK^{sep}}$ with self-intersection $-1$.
Moreover, by Corollary~\ref{corollary:markings-differ-sections} two such sections cannot intersect the same collection of curves among~\mbox{$E_1,\ldots,E_5$}. On the other hand, there are exactly $16$ ways to choose an odd number of curves among~\mbox{$E_1,\ldots,E_5$}. In other words, for any choice of
an odd number of curves among~\mbox{$E_1,\ldots,E_5$}, there is a unique section of $\phi_{\KK^{sep}}$ with self-intersection $-1$
which intersects these curves and no other curves among~\mbox{$E_1,\ldots,E_5$}.
\end{proof}

For a given marking of $S$ one can naturally obtain two equivalent markings of $\phi$, and for a given marking of $\phi$ one can obtain a marking of $S$.
This is done by the following two constructions.

\begin{construction}\label{construction:markings-1-to-1-S-to-X}
Let $S$ be a del Pezzo surface of degree $4$ over a field $\KK$, let~$X$ be a smooth cubic surface
obtained as a blow up $f\colon X\to S$ at a $\KK$-point $p$, and let~\mbox{$\phi\colon X\to\PP^1$} be the conic bundle defined by the $f$-exceptional line on $X$.
Let $L_1$, $\ldots$, $L_5$ be a marking of $S$, and let $Q\subset S_{\KK^{sep}}$ be the line meeting each of the lines $L_i$.
Observe that $p$ is not contained in any line on $S_{\KK^{sep}}$, because otherwise $X$ would not be a del Pezzo surface.
The proper transforms $f_*^{-1}Q$ and $f_*^{-1}L_i$ are lines on~\mbox{$X_{\KK^{sep}} \subset \PP^3$}. For any~$i$ the lines~$f_*^{-1}Q$ and~$f_*^{-1}L_i$
intersect each other. Therefore these lines are contained in some plane $\Pi_i$. The intersection of $\Pi_i$ with~$X_{\KK^{sep}}$ consists of three lines. Two of these lines are~$f_*^{-1}Q$
and~$f_*^{-1}L_i$, and we denote the third line by $E_i$. The image $f(E_i)$ is neither a line on~$S_{\KK^{sep}}$, nor a point. Therefore by
Remark~\ref{remark:lines-on-cubic} the line~$E_i$ is an irreducible component of a degenerate fiber of~\mbox{$\phi_{\KK^{sep}}\colon X_{\KK^{sep}}\to\PP^1$}. Moreover, for $i \neq j$ one has
\begin{multline*}
1 = -K_{X_{\KK^{sep}}} \cdot E_j = \left(f_*^{-1}Q + f_*^{-1}L_i + E_i\right) \cdot E_j =\\
= f_*^{-1}Q \cdot E_j + f_*^{-1}L_i \cdot E_j + E_i \cdot E_j = 1 +  E_i \cdot E_j.
\end{multline*}
This gives $E_i \cdot E_j = 0$. In other words, $E_i$ are irreducible components of pairwise different
fibers $E_i\cup F_i$ of $\phi_{\KK^{sep}}$. We will say that the (equivalent) markings
$E_1\cup F_1,\ldots,E_5\cup F_5$ and~\mbox{$F_1\cup E_1,\ldots,F_5\cup E_5$} of~$\phi$
\emph{correspond} to the marking
$L_1,\ldots,L_5$ of $S$.
\end{construction}

The converse to Construction~\ref{construction:markings-1-to-1-S-to-X}
is given as follows.

\begin{construction}\label{construction:markings-1-to-1-X-to-S}
Let $S$ be a del Pezzo surface of degree $4$ over a field $\KK$, let $X$ be a smooth cubic surface
obtained as a blow up $f\colon X\to S$ at a $\KK$-point, and let $\phi\colon X\to\PP^1$ be the conic bundle defined by the $f$-exceptional line on $X$.
Let $E_1 \cup F_1, \ldots, E_5 \cup F_5$ be a marking of $\phi$.
After replacing it by the equivalent marking if necessary, we may assume that
every $(-1)$-curve on $X_{\KK^{sep}}$ that is a section of~$\phi_{\KK^{sep}}$ intersects an odd number of curves among~\mbox{$E_1,\ldots,E_5$}, see
Corollary~\ref{corollary:odd-number-of-fibers}. By Lemma~\ref{lemma:unique-section-odd}, for any $i\in\{1,\ldots,5\}$
there is a unique section $A_i$ of~$\phi_{\KK^{sep}}$
with self-intersection $-1$ that intersects $E_i$
and does not intersect other curves among~\mbox{$E_1,\ldots,E_5$}.
Denote by $L_i$ the image of $A_i$ on $S_{\KK^{sep}}$.
The curves $A_i$ are disjoint by Lemma~\ref{lemma:intersection-of-sections};
also, they do not intersect the unique bi-section of~$\phi_{\KK^{sep}}$ with self-intersection $-1$,
see Remark~\ref{remark:lines-on-cubic}.
Hence the curves $L_i$ are disjoint, and one has~\mbox{$L_i^2=-1$}.
Thus we obtain a marking $L_1,\ldots,L_5$ of~$S$.
We will say that this marking \emph{corresponds} to the marking~\mbox{$E_1\cup F_1, \ldots, E_5\cup F_5$} of~$\phi$.
\end{construction}

It is straightforward to see that for any marking
of $\phi$ corresponding to a marking of $S$ via Construction~\ref{construction:markings-1-to-1-S-to-X},
the marking of~$S$ corresponding to it via Construction~\ref{construction:markings-1-to-1-X-to-S} coincides with the original marking of~$S$.
Similarly, for the marking of $S$ corresponding to a marking of $\phi$ via Construction~\ref{construction:markings-1-to-1-X-to-S},
the marking of $\phi$ corresponding to it via Construction~\ref{construction:markings-1-to-1-S-to-X}
is either the original marking of~$\phi$, or the marking equivalent to the original marking.
Thus, Constructions~\ref{construction:markings-1-to-1-S-to-X} and~\ref{construction:markings-1-to-1-X-to-S}
give mutually inverse one-to-one correspondences between markings of $S$ and pairs of equivalent markings of~$\phi$.

\begin{lemma}\label{lemma:markings-wd-equivariant}
Let $S$ be a del Pezzo surface of degree $4$ over a field $\KK$, let $X$ be a smooth cubic surface
obtained as a blow up $f\colon X\to S$ of $S$ at a $\KK$-point, and let~\mbox{$\phi\colon X\to\PP^1$} be the conic bundle defined by the $f$-exceptional line on $X$.
Consider a marking of~$\phi$, the corresponding marking of $S$, and the actions of $\WD$ on~$\Pic(X_{\KK^{sep}})$ and~$\Pic(S_{\KK^{sep}})$
given by these markings. Then the natural embedding
$f^*\colon \Pic(S_{\KK^{sep}}) \hookrightarrow \Pic(X_{\KK^{sep}})$ is $\WD$-equivariant.

Moreover, let the group $G$ be either
$\Gal(\KK^{sep}/\KK)$, or a subgroup of $\Aut(S)$. Suppose that~$X$
is obtained as a blow up of $S$ at a $G$-fixed point; thus, the group $G$ acts on $X$ as well.
The action of~$G$ on $\Pic(X_{\KK^{sep}})$ and $\Pic(S_{\KK^{sep}})$ gives homomorphisms $\pi_X \colon G \to \WD$ and~\mbox{$\pi_S \colon G \to \WD$}. Then the following diagram is commutative:
$$
\xymatrix{
&G\ar@{->}[ld]_{\pi_S}\ar@{->}[rd]^{\pi_X}&\\
\WD\ar@{->}[rr]^{\sim}&& \WD
}
$$
\end{lemma}

\begin{proof}
Let $E_1 \cup F_1, \ldots, E_5 \cup F_5$ be a marking of $\phi$.
After replacing it by the equivalent marking
if necessary, we may assume that
every $(-1)$-curve on $X_{\KK^{sep}}$ that is a section of~$\phi_{\KK^{sep}}$ intersects an odd number of curves among~\mbox{$E_1,\ldots,E_5$}, see
Corollary~\ref{corollary:odd-number-of-fibers}.
Recall from Remark~\ref{remark: equivalent-markings-gives-same-wd-actions}
that equivalent markings of~$\phi$ gives the same action of~$\WD$ on~$\Pic(X_{\KK^{sep}})$.

Let $L_1,\ldots,L_5$ be the marking of $S$ corresponding to the above marking of $\phi$ as in Construction~\ref{construction:markings-1-to-1-X-to-S},
and let~$Q$ be the unique line on $S_{\KK^{sep}}$ intersecting $L_1,\ldots,L_5$.
According to Lemma~\ref{lemma:unique-section-odd},
there is a unique section $A$ of $\phi_{\KK^{sep}}$ with self-intersection~$-1$ that intersects all the curves $E_i$.
By Lemma~\ref{lemma:intersection-of-sections} the section $A$ intersects
each of the curves~$f_*^{-1}L_i$ (which are also sections of $\phi_{\KK^{sep}}$). Therefore, $A$ coincides with~$f_*^{-1}Q$.
Note that for any~\mbox{$\sigma \in \HHH \subset \WD$} one has $\sigma(E_i) = E_{\sigma(i)}$, and therefore $\sigma(f_*^{-1}Q)$ intersects all the curves~$E_i$.
Thus, we know that~\mbox{$\sigma(f_*^{-1}Q) = f_*^{-1}Q$}, which means that $f_*^{-1}Q$ is $\HHH$-invariant.

Consider the action of $\WD$ on $\Pic(X_{\KK^{sep}})$.
One has $\iota_{2345}(E_1) = E_1$ and~\mbox{$\iota_{2345}(E_j) = F_j$} for~\mbox{$j \neq 1$}. Therefore
$$
\iota_{2345}(f_*^{-1}Q) = f_*^{-1}L_1.
$$
Similarly we can show that~\mbox{$\iota_{klmn}(f_*^{-1}Q) = f_*^{-1}L_i$}, where $\{i, k, l, m, n\} = \{1, 2, 3, 4, 5\}$. Moreover, for any $\sigma \in \HHH\cong\SSS_5$ one has
$$
\sigma(f_*^{-1}L_i) = \sigma\iota_{klmn}(f_*^{-1}Q) = \sigma\iota_{klmn}\sigma^{-1} (\sigma(f_*^{-1}Q)) = \iota_{\sigma(k)\sigma(l)\sigma(m)\sigma(n)} (f_*^{-1}Q) = f_*^{-1}L_{\sigma(i)}.
$$

Consider the action of $\WD$ on $\Pic(S_{\KK^{sep}})$. We have~\mbox{$\sigma(Q) = Q$} and $\sigma (L_i) = L_{\sigma(i)}$ for any~\mbox{$\sigma \in \HHH$}, and $\iota_{klmn}(Q) = L_i$ for $\iota_{klmn} \in \NNN$, where~\mbox{$\{i, k, l, m, n\} = \{1, 2, 3, 4, 5\}$}. We see that this action coincides with the action of~$\WD$ on $\Pic(S_{\KK^{sep}})$ induced by the embedding~$f^*$. Therefore $f^*$
is $\WD$-equivariant.
This proves the first assertion of the lemma.

Also, the embedding $f^*$ is obviously $G$-equivariant. It means that for any $g \in G$ the action of $\pi_X(g)$ on $\Pic(S_{\KK^{sep}})$ induced by $f^*$ coincides with the action of $\pi_S(g)$.
Since~$f^*$ is $\WD$-equivariant, we have $\pi_X(g) = \pi_S(g)$, and the second assertion of the lemma holds.
\end{proof}

\section{Pencil of quadrics}
\label{section:pencil}

Let $\KK$ be an arbitrary field, and let $S$ be a del Pezzo surface of degree $4$ over $\KK$. Then
the anticanonical linear system of $S$ provides an embedding
$S\hookrightarrow \PP^4$, which represents $S$ as an intersection of two quadrics in~$\PP^4$. Denote
by~$\cP$ the pencil of quadrics in $\PP^4$ passing through $S$. Thus, one has~\mbox{$\cP\cong\PP^1$}.
The following description of singular quadrics in the pencil~$\cP$ is well known;
in the sequel we will mostly use it without explicit reference.

\begin{theorem}\label{theorem:dP4-basic}
Let $S$ be a del Pezzo surface of degree $4$ over a field $\KK$. The following assertions hold.
\begin{itemize}
\item[(i)] The singular quadrics in the pencil $\cP$ are parameterized by
a smooth closed subscheme $\Delta(S)$ of length $5$.

\item[(ii)] If $R\in\cP$ is a singular quadric defined over $\KK$, then~$R$
is a cone over a smooth quadric surface with a vertex at a $\KK$-point.

\item[(iii)] Let $R$ be a singular quadric in $\cP$.
Then the projection from its vertex defines a separable double cover
of $S$ to the base of the cone~$R$.
\end{itemize}
\end{theorem}

\begin{proof}
If the characteristic of $\KK$ is different from $2$, assertions (i) and~(ii) are given by~\mbox{\cite[Proposition~2.1]{Reid}}.
If the characteristic equals~$2$, then assertion~(i)
is given by~\mbox{\cite[Theorem~3.2]{DD}}. Furthermore, in this case
the quadric $R$ has corank~$1$ by~\mbox{\cite[Corollary~3.3]{DD}}, i.e. the kernel of the
associated bilinear form is one-dimensional. The corresponding $\KK$-point of $\PP^4$ is the (unique)
singular point of the quadric $R$: indeed, since $R$ is singular, the half-discriminant of the corresponding quadratic
form $q_R$ vanishes, and the latter is computed as the value of $q_R$ on a vector generating the kernel
of the bilinear form
$$
b_R(u,v)=q_R(u+v)-q_R(u)-q_R(v),
$$
see~\mbox{\cite[\S2.2]{DD}} for details. This proves assertion~(ii) in the case when
the characteristic of~$\KK$ equals~$2$.

Let $R$ be a singular quadric in $\cP$; we know from assertion~(ii) that it is a cone over
a smooth quadric surface~$T$. Let $\omega\colon S\to T$ be the double cover
defined by the projection from the vertex of~$R$.
Let us prove that $\omega$ is geometrically separable. We can assume that $\KK$ is algebraically closed.
If $\omega$ is not separable, then it is purely inseparable (and the characteristic of $\KK$ equals~$2$).
Thus, it is a homeomorphism in \'etale topology. This is impossible, because
$$
\rk H^2_{\et}(S,\ZZ_\ell)=\rkPic(S)=6\neq 2=\rkPic(T)=\rk H^2_{\et}(T,\ZZ_\ell),
$$
where $\ell\neq 2$ is any prime (cf. \cite[Proposition~5.1]{DM}).
Thus, $\omega$ is geometrically separable, and in particular separable. This proves assertion~(iii).
\end{proof}

\begin{corollary}
Let $S$ be a del Pezzo surface of degree $4$ over a field $\KK$.
Then all the points of $\Delta(S)_{\bar{\KK}}$, as well as the vertices
of the singular quadrics in~$\cP$, are
defined over~$\KK^{sep}$.
\end{corollary}

\begin{proof}
The assertion about $\Delta(S)$ follows from the fact that an irreducible variety remains irreducible after a purely non-separable extension,
see e.g.~\mbox{\cite[Proposition 2.7(c)]{Liu}}.
In particular, we see that all the singular quadrics in $\cP$ are defined over~$\KK^{sep}$.
Thus, the assertion about the vertices of the singular quadrics is given by Theorem~\ref{theorem:dP4-basic}(ii).
\end{proof}

\begin{remark}
If $R$ is a quadric in an odd-dimensional projective space over a field $\KK$
such that~$R_{\bar{\KK}}$ is a cone over a smooth quadric with vertex at a point, then
the vertex may be defined
over a purely inseparable extension of $\KK$, but not over $\KK$ itself (and not over~$\KK^{sep}$).
For instance, set $\KK=\mathbb{F}(t)$, where
$\mathbb{F}$ is any field of characteristic $2$, and $t$ is a transcendental variable.
Let $n\ge 4$ be an even integer, and let $R$ be a quadric over $\KK$
defined by equation
$$
x_1x_2+\ldots+x_{n-3}x_{n-2}+x_{n-1}^2+tx_n^2=0
$$
in the projective space $\PP^{n-1}$ with homogeneous coordinates $x_1,\ldots,x_n$.
Then~$R_{\bar{\KK}}$ is a cone over a smooth quadric with vertex at the point
$$
p=(0:\ldots:0:\sqrt{t}:1)
$$
which is defined over a purely inseparable extension of $\KK$.
For an alternative geometric construction of an example in~$\PP^3$, we refer the reader to~\mbox{\cite[Example 7.14]{BFSZ}}.
\end{remark}

Let $R\subset \PP^4$ be a singular quadric containing a del Pezzo surface $S$ of degree $4$.
According to Theorem~\ref{theorem:dP4-basic}(iii),
the double cover $\omega\colon S\to T$ defined by the projection from the vertex of~$R$ gives
an involution $\iota_R$ on~$S$.
We refer the reader to~\mbox{\cite[\S7]{DD}} for a slightly different description
of these involutions.

\begin{remark}\label{remark:involution-geometric-description}
The involution $\iota_R$ has the following geometric description.
For a point $p$ on~$S$ consider the line $A$ passing through $p$ and the vertex of $R$. If $A$ transversally intersects~$S$ at $p$, then $\iota_R(p)$ is the other point of intersection of $A$ and $S$. Otherwise $A$ lies in the tangent plane of $S$ at the point $p$ and $\iota_R(p)=p$.
In particular, if $\Pi$ is any linear subspace of~$\PP^4$ passing through the vertex of
$R$, then the intersection $S\cap\Pi$ is invariant under the involution~$\iota_R$.
\end{remark}

Let us make some observations concerning the action of the involutions
$\iota_R$ on the lines on a del Pezzo surface of degree~$4$.

\begin{lemma}\label{lemma:involution-description}
Let $S$ be a del Pezzo surface of degree $4$ over a separably closed field, and let $Q\subset S$ be a line.
The following assertions hold.
\begin{itemize}
\item[(i)]
If $R$ is a singular quadric containing $S$, then $\iota_R(Q)$
is a line which meets $Q$ at a point.

\item[(ii)]
If $R_1$ and $R_2$ are two different singular quadrics containing $S$, then
$$
\iota_{R_1}(Q)\neq\iota_{R_2}(Q).
$$

\item[(iii)]
For each line $L\subset S$ intersecting $Q$ there exists a unique singular quadric $R$ containing $S$
such that $L = \iota_{R}(Q)$.
\end{itemize}
\end{lemma}
\begin{proof}
Let $R$ be a singular quadric containing $S$.
The image $\iota_R(Q)$ is a line.
Suppose that~\mbox{$Q = \iota_R(Q)$}. From Remark~\ref{remark:involution-geometric-description} it is easy to see that if a line on $S$ is $\iota_R$-invariant,
then~$\iota_R$ fixes every point on this line.
In particular, $\iota_R$ fixes every point on $Q$.
Therefore for any line $L\subset S$ intersecting $Q$ one has $\iota_R(L) = L$, because for every point of $S$ there are at most two lines
passing through this point. Thus $\iota_R$ also fixes every point on $L$. In the same way we can show that all $16$ lines on $S$ are fixed by $\iota_R$. Therefore $\iota_R$ acts trivially on $S$ by Lemma~\ref{lemma:Weyl-group-on-dP4-basic}(iii), and we obtain a contradiction.
This shows that~\mbox{$\iota_R(Q)\neq Q$}.

Now consider the plane $\Xi$ passing through $Q$ and the vertex of $R$.
The intersection $\Xi \cap S$ is $\iota_R$-invariant, and contains the line $Q$. Therefore it also contains the line $\iota_R(Q)\neq Q$.
Hence these two lines meet each other at a point.
This proves assertion~(i).

As we have just seen, the plane $\Xi$ spanned by the lines $Q$ and $\iota_R(Q)$ contains the vertex of~$R$. Therefore $\Xi$ is contained in $R$.
Thus if for two different singular quadrics $R_1$ and~$R_2$ containing $S$ one has $\iota_{R_1}(Q) = \iota_{R_2}(Q)$,
then $S = R_1 \cap R_2$ contains a plane, which is impossible.
This proves assertion~(ii).

Let $R_1,\ldots, R_5$ be the five singular quadrics containing $S$.
According to assertion~(i), the lines $L_i=\iota_{R_i}(Q)$ are different from $Q$ and intersect $Q$;
according to assertion~(ii), they are pairwise different. Now assertion~(iii) follows from the fact that
there are exactly five lines on $S$ which intersect~$Q$.
\end{proof}

\medskip
Lemma~\ref{lemma:involution-description}
provides a natural bijection between singular quadrics
containing a del Pezzo surface~$S$ of degree~$4$, and lines
on $S$ intersecting a given line $Q\subset S$.
The next observation extends this to an isomorphism of two copies of~$\PP^1$
which identifies the corresponding subschemes of length~$5$.
This allows one to characterize $\Delta(S)$ in terms of the configuration of lines on $S$
provided that one of the lines is defined over the base field.

\begin{lemma}\label{lemma:5-lines}
Let $S$ be a del Pezzo surface of degree $4$ over a field $\KK$.
Suppose that~$S$ contains a line~$Q$. Let $\Lambda_Q(S)$ be the
Hilbert scheme of lines on $S$ different from $Q$ and intersecting $Q$. Consider
$\Lambda_Q(S)$ as a closed subscheme of $Q\cong\PP^1$ by assigning to a line in~\mbox{$\Lambda_Q(S)$} its
intersection point with $Q$. Then
there exists an isomorphism
$$
Q\stackrel{\sim}\longrightarrow \mathcal{P}
$$
which maps $\Lambda_Q(S)$ isomorphically to $\Delta(S)$.
\end{lemma}

\begin{proof}
Let $\mathcal{H}\cong\PP^2$ be the linear system of hyperplanes in $\PP^4$ passing through the line~$Q$.
We observe that for every $p\in Q$ and every $R\in\mathcal{P}$, the tangent hyperplane~\mbox{$T_p R$} is
contained in $\mathcal{H}$. Conversely, every hyperplane in $\mathcal{H}$ has such a form for some $p$ in $R$
by dimension count.

Consider the surface $Q\times\mathcal{P}\cong\PP^1\times\PP^1$, and the morphism
$$
\tau\colon Q\times\mathcal{P}\to \mathcal{H}, \quad (p, R)\mapsto T_p R.
$$
More explicitly, if $q_R$ is the quadratic form which defines a quadric $R$, and
$$
b_R(u,v)=q_R(u+v)-q_R(u)-q_R(v)
$$
is the corresponding bilinear form, then the hyperplane $T_p R$ is given by the linear equation
$$
b_R(p, \cdot)=0,
$$
whose coefficients are linear in the coefficients of $q_R$ and in the coordinates of $p$.
In other words, $\tau$ is defined by polynomials of bidegree $(1,1)$.
This implies that $\tau$ is a finite morphism of degree $2$ branched over a smooth conic in $\mathcal{H}\cong\PP^2$,
and ramified in a smooth divisor $\mathcal{T}\subset Q\times\mathcal{P}$ of bidegree $(1,1)$.

The divisor $\mathcal{T}$ provides a natural isomorphism between the rulings $Q$ and $\mathcal{P}$ of $Q\times\mathcal{P}$.
It remains to check that $\Lambda_Q(S)$ is mapped to $\Delta(S)$ under this isomorphism.
To do this, we may assume that
$\KK$ is separably closed.
Choose a point~\mbox{$p\in\Lambda_Q(S)$}. In other words, $p\in Q$ is a point such that there exists a line $L\subset S$ different
from $Q$ and passing through~$p$.
By Lemma~\ref{lemma:involution-description}(iii), there exists a unique singular quadric $R$ containing $S$ such that
$$
\iota_{R}(Q) = L.
$$
The plane $\Xi$ spanned by $Q$ and $L$ is contained in $R$. Note that the intersection $R\cap T_p R$ consists of two planes $\Xi$ and $\Xi'$,
and $\Xi \cap \Xi'$ is the line passing through the point~$p$ and the vertex of the cone $R$.
The latter line intersects $S$ only at the point $p$, because~\mbox{$\iota_R(p)=p$}. 
Therefore, $C = \Xi' \cap S$ is a conic such that $Q \cap C=p$.
This shows that the hyperplane section of $S$ by $H=T_p R$ is singular only at the point $p$,
which in turn means that $H$ is not tangent to any quadric from $\mathcal{P}$ at any point
of $Q$ different from~$p$. Hence
$$
(p, R)\in \mathcal{T},
$$
so that $p$ is mapped to $R$ by the above isomorphism.
Since $\Lambda_Q(S)$ and $\Delta(S)$ are smooth finite schemes of the same length, we conclude that $\Lambda_Q(S)$
isomorphically mapped to~$\Delta(S)$.
\end{proof}

\begin{remark}
In the notation of Lemma~\ref{lemma:5-lines}, assume that the field $\KK$ is separably closed, and
let~\mbox{$\pi\colon S\to\PP^2$} be the contraction of the five lines intersecting $Q$.
Then $\pi(Q)$ is a conic, and for any $H\in\mathcal{H}$ one has $\pi(H\cap S)=Q\cup\ell$
for some line $\ell$.
This gives an identification of $\mathcal{H}$ with the dual projective plane $(\PP^2)^\vee$.
Under such an identification, the branch locus of the double cover $\tau$
is the conic $Q^\vee$ projectively dual to $Q$, i.e. the locus of lines on $\PP^2$ tangent to $Q$,
or, in other words, the locus of lines
such that the corresponding hyperplane section of $S$ is singular only at one point on $Q$.
\end{remark}

\begin{remark}\label{remark:Skorobogatov-upgrade}
Lemma~\ref{lemma:5-lines} was established in course of the proof of \cite[Lemma~2.2]{Skorobogatov-Kummer}
under the additional assumption that the field $\KK$ contains at least five elements.
This assumption was necessary because the construction of the isomorphism depended on a choice
of a point outside $\Lambda_Q(S)$ (although a posteriori one could see that the construction
did not actually depend on the point). Our approach allows to avoid the choice of a point,
and thus to drop the assumption on the number of elements
in the field.
\end{remark}

\section{Biregular classification}
\label{section:biregular}

Following~\cite{Skorobogatov-Kummer}, we introduce some additional notions and notation associated with del Pezzo surfaces of degree~$4$.

Consider the five-dimensional \'etale $\KK$-algebra $\LL=\KK[\Delta(S)]$, and set
$$
\tilde{G}(S)=R_{\LL/\KK}\,(\ZZ/2\ZZ),
$$
where $\ZZ/2\ZZ$ is considered as a (constant) group scheme over $\LL$.
Then $\tilde{G}(S)$ is a group scheme over $\KK$ such that $\tilde{G}(S)_{\LL}\cong(\ZZ/2\ZZ)^5$.
Furthermore, one has an isomorphism
$$
(\ZZ/2\ZZ)\otimes_{\Spec\KK}\Spec\LL\stackrel{\sim}\longrightarrow \ZZ/2\ZZ
$$
of group schemes over $\LL$, where $\ZZ/2\ZZ$ on the left hand side is considered as a group scheme over $\KK$,
and $\ZZ/2\ZZ$ on the right hand side is considered as a group scheme over $\LL$.
Since Weil restriction of scalars is the right adjoint functor to base change,
we obtain an injective homomorphism of groups schemes $\ZZ/2\ZZ \hookrightarrow \tilde{G}(S)$ over $\KK$.
Denote
$$
G(S)=\tilde{G}(S)/(\ZZ/2\ZZ).
$$
Thus, $G(S)$ is a group scheme over $\KK$ such that $G(S)_{\LL}\cong(\ZZ/2\ZZ)^4$.

Let $\NNN\cong (\ZZ/2\ZZ)^4$ be the normal subgroup of $\WD$ endowed with an action  of the Galois group $\Gal(\KK^{sep}/\KK)$ by conjugation.
Denote by $\Lambda(S)$ the Hilbert scheme of lines on~$S$.

\begin{lemma}\label{lemma:GS-vs-NNN}
Let $\KK$ be a field.
Let $S$ be a del Pezzo surface of degree $4$ over $\KK$.
Let~$G(S)$ and $\Lambda(S)$ be constructed as above.
The following assertions hold.
\begin{itemize}
\item[(i)] There is a
$\Gal(\KK^{sep}/\KK)$-equivariant
isomorphism
$$
G(S)_{\KK^{sep}}\stackrel{\sim}\longrightarrow \NNN.
$$

\item[(ii)] There is a natural action of the group scheme $G(S)$ on $S$.

\item[(iii)]
The Hilbert scheme of lines $\Lambda(S)$ is a torsor over $G(S)$.
\end{itemize}
\end{lemma}

\begin{proof}
Let $R\subset\PP^4$ be a singular quadric containing $S_{\KK^{sep}}$, so that $R$
is a cone over a smooth quadric surface, and~$S$ is a separable double cover of the base of this cone.
Let
$$
\cG(S)\subset\Aut(S_{\KK^{sep}})
$$
be the group generated by the five Galois involutions
of such double covers, that is, by the involutions $\iota_R$ introduced in Section~\ref{section:pencil}.

Choose a marking $L_1,\ldots,L_5$ of $S$, and consider the action of $\WD$ on $\Pic(S_{\KK^{sep}})$ given by this marking.
Thus, we get an embedding $\Aut(S_{\KK^{sep}}) \hookrightarrow \WD$, see Lemma~\ref{lemma:Weyl-group-on-dP4-basic}(iii).
This endows~$\cG(S)$ with an action of the Galois group $\Gal(\KK^{sep}/\KK)$ by conjugation, via
the natural homomorphism
$$
\Gal(\KK^{sep}/\KK)\to\WD.
$$

Let us show that the embedding $\Aut(S_{\KK^{sep}})\hookrightarrow \WD$ provides a
$\Gal(\KK^{sep}/\KK)$-equivariant
isomorphism
$$
\cG(S)\stackrel{\sim}\longrightarrow \NNN.
$$
Consider an involution $\iota_R$ corresponding to a singular quadric~\mbox{$R\subset\PP^4$} containing $S$.
The involution $\iota_R$ maps any line $L$ to a line $\iota_R(L)\neq L$ meeting~$L$ at a point by Lemma~\ref{lemma:involution-description}(i).
Note that the group $\cG(S)$ contains at least five elements with such property.

On the other hand, it is easy to see that any element of order $2$ in $\WD$ is conjugate to $\iota_{12}$, $\iota_{1234}$, $(12)$, $\iota_{34}(12)$, or $(12)(34)$. One has
$$
\iota_{12} L_1 = L_2, \quad (12) L_1 = L_2, \quad \iota_{34}(12)L_3 = L_4, \quad (12)(34)L_1 = L_2.
$$
Therefore for any element $g$ of order $2$ that is not conjugate to $\iota_{1234}$ one can find a pair of disjoint lines interchanged by $g$.
Observe that there are exactly five elements in $\WD$ conjugate to $\iota_{1234}$. Thus $\cG(S)$ contains all these elements, and therefore $\cG(S) \subset \WD$ coincides with $\NNN$, since $\NNN$ is generated by the elements conjugate to $\iota_{1234}$.

Now consider a finite group
$$
\tilde{\cG}(S)=\prod\limits_{R\in \Delta(S)_{\KK^{sep}}}\ZZ/2\ZZ
$$
with an action of the Galois group $\Gal(\KK^{sep}/\KK)$, and the $\Gal(\KK^{sep}/\KK)$-equivariant homomorphism $\tilde{\cG}(S) \rightarrow \cG(S)$, such that the
generator of the factor corresponding to~\mbox{$R\in \Delta(S)_{\KK^{sep}}$} maps to the involution $\iota_R$ corresponding to the singular quadric~$R$. Note that this homomorphism is surjective, since $\cG(S)$ is generated by the involutions~$\iota_R$. Moreover, the kernel of this homomorphism is isomorphic to $\ZZ/2\ZZ$ generated by the product of the five generators of the factors of $\tilde{\cG}(S)$. To show this we can consider $\cG(S)$ as a subgroup of~$\WD$, then the images of the five generators are elements conjugate to~$\iota_{1234}$, but
$$
\iota_{1234} \iota_{1235} \iota_{1245} \iota_{1345} \iota_{2345} = 1.
$$

Now consider a $\Gal(\KK^{sep}/\KK)$-equivariant isomorphism
$$
\tilde{G}(S)_{\KK^{sep}}\cong \tilde{\cG}(S).
$$
Observe that $G(S)_{\KK^{sep}}\cong\tilde{G}(S)_{\KK^{sep}}/(\ZZ/2\ZZ)$ and $\cG(S)\cong\tilde{\cG}(S)/(\ZZ/2\ZZ)$, where the subgroups
$\ZZ/2\ZZ$ are generated by the product of generators of the factors of~$\tilde{G}(S)_{\KK^{sep}}$ and~$\tilde{\cG}(S)$ respectively.
This defines a $\Gal(\KK^{sep}/\KK)$-equivariant isomorphism~\mbox{$G(S)_{\KK^{sep}}\cong\cG(S)$}.
Together with the isomorphism~\mbox{$\cG(S) \cong \NNN$} this proves assertion~(i).

To prove assertion~(ii), observe that the category of smooth
finite group schemes over a field~$\KK$ is equivalent to the category
of finite groups over~$\KK^{sep}$ with an action of the Galois group~$\Gal(\KK^{sep}/\KK)$.

Assertion~(iii) follows from assertion~(i) and Lemma~\ref{lemma:Weyl-group-on-dP4-basic}(ii).
\end{proof}

Denote by $S_{\Delta(S)}^{triv}$ the del Pezzo surface of degree $4$ obtained as a blow up of $\PP^2$ at the image of $\Delta(S)$
under the Veronese embedding $v_2\colon\cP\to\PP^2$ of the pencil $\cP \cong \PP^1$ of quadrics passing through $S$.

\begin{lemma}[{see \cite[Lemma~2.2]{Skorobogatov-Kummer}}]
\label{lemma:Delta-vs-Delta-triv}
Let $S$ be a del Pezzo surface of degree $4$ over a field $\KK$. The following assertions hold.
\begin{itemize}
\item[(i)] The surface $S_{\Delta(S)}^{triv}$ contains a line defined over $\KK$.

\item[(ii)] If $S$ contains a line defined over $\KK$, then $S\cong S_{\Delta(S)}^{triv}$.

\item[(iii)] One has
$$
\Delta(S_{\Delta(S)}^{triv})\cong\Delta(S)
$$
as subschemes of $\PP^1$.
\end{itemize}
\end{lemma}

\begin{proof}
The proper transform of the conic $v_2(\cP)$ is a line on $S_{\Delta(S)}^{triv}$; this proves assertion~(i).
Assertion~(iii) follows from assertion~(i) and Lemma~\ref{lemma:5-lines}.

Suppose that $S$ contains a line $Q$ defined over $\KK$. Then the quintuple $\mathcal{L}$ of lines on~$S_{\KK^{sep}}$ different from $Q$ and intersecting $Q$
is defined over $\KK$, and the lines of $\mathcal{L}$ are pairwise disjoint. Let $\pi\colon S\to\PP^2$
be the contraction of $\mathcal{L}$. Then $\pi(Q)$ is a conic on $\PP^2$, and
$$
\pi(Q)\cong Q\cong\PP^1.
$$
Finally, $\pi(\mathcal{L})\subset \pi(Q)$ is
isomorphic to $\Delta(S)$ as a subscheme of $\pi(Q)\cong\PP^1$ by Lemma~\ref{lemma:5-lines}.
This proves assertion~(ii).
\end{proof}

Lemma~\ref{lemma:Delta-vs-Delta-triv}(iii)
implies that
$$
G(S_{\Delta(S)}^{triv})\cong G(S).
$$
In particular, $G(S)$ acts on $S_{\Delta(S)}^{triv}$.

Recall that if $\Lambda$ is a (left) torsor over a group scheme $G$, and $X$ is a scheme with a (left) action of~$G$, then the
\emph{twist} of $X$ by $\Lambda$ is defined as
$$
\leftidx{^\Lambda}{X}{}=(\Lambda\times X)/G,
$$
where the action of $G$ on $X\times\Lambda$ is diagonal; in other words, the quotient is taken by the equivalence
relation~\mbox{$(g^{-1}\lambda,x)\sim (\lambda, gx)$}.

\begin{remark}
Observe that every non-trivial element of $G(S)_{\KK^{sep}}$ has order $2$.
Thus, every non-trivial torsor
$\Lambda$ over $G(S)$ has order $2$ as an element of the cohomology
group~\mbox{$H^1\big(\Gal(\KK^{sep}/\KK), G(S)_{\KK^{sep}}\big)$}.
In particular, the torsor $\leftidx{^\Lambda}{\Lambda}{}$ is trivial, and
for any scheme~$X$ with an action of $G(S)$ one has
$$
\phantom{X}^\Lambda(\leftidx{^\Lambda}{X}{})\cong X.
$$
\end{remark}

The following theorem of A.\,Skorobogatov provides a biregular classification of
del Pezzo surfaces of degree~$4$.

\begin{theorem}[{see \cite[Theorem~2.3]{Skorobogatov-Kummer}}]
\label{theorem:Skorobogatov}
Let $S$ be a del Pezzo surface of degree $4$ over a field $\KK$.
Let $\Delta(S)$, $G(S)$, $\Lambda(S)$, and~$S_{\Delta(S)}^{triv}$ be constructed as above.
Then $S$ is isomorphic to the twist of $S_{\Delta(S)}^{triv}$ by the $G(S)$-torsor $\Lambda(S)$.
In particular, the surface $S$ is uniquely defined by the isomorphism class of the subscheme $\Delta(S)\subset\PP^1$
and the isomorphism class of the $G(S)$-torsor $\Lambda(S)$.
\end{theorem}

\begin{proof}
Let
$$
S'=\leftidx{^{\Lambda(S)}}{S}{}
$$
be the twist of $S$ by the $G(S)$-torsor $\Lambda(S)$. Then
$$
\Lambda(S')\cong \leftidx{^{\Lambda(S)}}{\Lambda(S)}{}.
$$
Hence $\Lambda(S')$ is the trivial $G(S)$-torsor.
Also, observe that $\Delta(S')$ and $\Delta(S)$ are isomorphic as subschemes of~$\PP^1$, since the action of $G(S)$ on $\cP$ and $\Delta(S)$ is trivial.
Therefore, one has
$$
S'\cong S_{\Delta(S')}^{triv}\cong S_{\Delta(S)}^{triv}
$$
by Lemma~\ref{lemma:Delta-vs-Delta-triv}(ii).
Thus, we obtain
$$
S\cong \phantom{S}^{\Lambda(S)}(\leftidx{^{\Lambda(S)}}{S}{})\cong
\leftidx{^{\Lambda(S)}}{S'}{}\cong \leftidx{^{\Lambda(S)}}{S_{\Delta(S)}^{triv}}{}.
\qedhere
$$
\end{proof}

As a by-product of Theorem~\ref{theorem:Skorobogatov}, one obtains Theorem~\ref{theorem:Skorobogatov-short}.
Furthermore, there is the following convenient corollary of Theorem~\ref{theorem:Skorobogatov}.

\begin{corollary}\label{corollary:Skorobogatov}
Let $S$ and $S'$ be del Pezzo surfaces of degree $4$ over a field $\KK$.
Suppose that one has an automorphism $\PP^1\stackrel{\sim}\longrightarrow \PP^1$ which induces
an isomorphism $\Delta(S)\stackrel{\sim}\longrightarrow \Delta(S')$, and one has an isomorphism
$$
\tilde{\psi}\colon \Lambda(S)\stackrel{\sim}\longrightarrow \Lambda(S')
$$
of torsors over $G(S)\cong G(S')$. Then there exists an isomorphism
$$
\psi\colon  S\stackrel{\sim}\longrightarrow S'
$$
such that the induced isomorphism $\Lambda(S)\to \Lambda(S')$
coincides with $\tilde{\psi}$.
\end{corollary}

\begin{proof}
Recall that the isomorphism $\Delta(S)\stackrel{\sim}\longrightarrow \Delta(S')$ gives an isomorphism
of group schemes $G(S)\stackrel{\sim}\longrightarrow G(S')$.
We know from Theorem~\ref{theorem:Skorobogatov} that both $S$ and $S'$ are obtained as
the twist of the surface $S_{\Delta(S)}^{triv}$ by the torsor $\Lambda(S)\cong\Lambda(S')$ over $G(S)\cong G(S')$.
Thus, we have an isomorphism $\psi\colon S\stackrel{\sim}\longrightarrow S'$, and an induced isomorphism
$$
\hat{\psi}\colon \Lambda(S)\stackrel{\sim}\longrightarrow \Lambda(S').
$$
By construction, $\hat{\psi}$
is the isomorphism of torsors $\leftidx{^{\Lambda(S)}}{\Lambda^{triv}}{}$ and $\leftidx{^{\Lambda(S')}}{\Lambda^{triv}}{}$
obtained from the isomorphism $\tilde{\psi}$, where $\Lambda^{triv}$ is the trivial torsor over
$G(S)\cong G(S')$. Therefore, one has~\mbox{$\hat{\psi}=\tilde{\psi}$}.
\end{proof}

We point out that a del Pezzo surface $S$ of degree $4$ cannot be recovered from
the $G(S)$-torsor $\Lambda(S)$ and the scheme $\Delta(S)$ without taking into account
the embedding $\Delta(S)\hookrightarrow \PP^1$.

\begin{example}
Let $\KK$ be a separably closed field, and let $S$ and $S'$ be del Pezzo surfaces of degree $4$ over $\KK$.
Then $\Delta(S) \cong \Delta(S')$ is just a quintuple of points over $\KK$, one has
$$
G(S)\cong G(S')\cong(\ZZ/2\ZZ)^4,
$$
and $\Lambda(S) \cong \Lambda(S')$ is the trivial torsor over~$(\ZZ/2\ZZ)^4$.
However, there are many non-isomorphic del Pezzo surfaces of degree~$4$ over~$\KK$.
\end{example}

\begin{remark}
One can show that the scheme $\Delta(S)$ (without the embedding into~$\PP^1$)
can be recovered from $\Lambda(S)_{\KK^{sep}}$ considered as a set with incidence relation and an action of the group $\Gal(\KK^{sep}/\KK)$.
Indeed, the points of $\Delta(S)_{\KK^{sep}}$ are in a one-to-one correspondence
with the five ways to split the $16$ lines of~$\Lambda(S)_{\KK^{sep}}$
into two eightuples, such that each eightuple consists
of four pairs of lines $L_1, L_1',\ldots,L_4,L_4'$ with~\mbox{$L_i\cap L_i'\neq\varnothing$}
and
$$
L_i\cap L_j=L_i\cap L_j'=L_i'\cap L_j'=\varnothing
$$
for $i\neq j$. Namely, these eightuples consist of irreducible components
of singular conics cut out on $S_{\KK^{sep}}$ by the planes lying in one of the two
families of planes on a singular quadric containing~$S_{\KK^{sep}}$. Therefore the five pairs of eightuples correspond
to five singular quadrics containing~$S_{\KK^{sep}}$. The action of $\Gal(\KK^{sep}/\KK)$ on $\Lambda(S)_{\KK^{sep}}$ induces an action of~\mbox{$\Gal(\KK^{sep}/\KK)$} on the five pairs of eightuples, which coincides with the action on $\Delta(S)_{\KK^{sep}}$.
\end{remark}

Similarly, a del Pezzo surface $S$ of degree $4$ cannot be recovered
from the scheme~\mbox{$\Delta(S)\subset\PP^1$} and the scheme $\Lambda(S)$ without
the action of $G(S)$, even if we take into account the natural
incidence relation on $\Lambda(S)_{\KK^{sep}}$.

\begin{example}
Let $\KK = \mathbb{R}$ be the field of real numbers, and let
$p_1$, $p_2$, $p_3$, $p_4$, and~$p_5$ be five points lying on a smooth conic in $\PP^2$.
Suppose that $p_1,\ldots,p_5$ are sufficiently general, i.e.
there are no non-trivial automorphisms of the conic mapping the set~\mbox{$\{p_1,\ldots,p_5\}$} to itself.
Consider the blowup $\pi\colon S_0 \to \PP^2$ of the points $p_i$, and
put $L_i = \pi^{-1}(p_i)$. Then~\mbox{$L_1, \ldots, L_5$} is a marking
of the del Pezzo surface $S_0$ of degree $4$. This marking allows one to fix
an action of $\WD$ on the Hilbert scheme of lines $\Lambda(S_0)$. Note that all lines
on $S_0$ are defined over~$\mathbb{R}$. Therefore, the Galois group
$\Gal(\mathbb{C}/\mathbb{R})$ acts trivially on $\Lambda(S_0)_{\mathbb{C}}$.
Moreover, the Galois group acts trivially on $G(S_0)_{\mathbb{C}} \cong (\ZZ/2\ZZ)^4$,
since all five points of $\Delta(S_0)_{\mathbb{C}}$ are defined over~$\mathbb{R}$,
cf.~Lemma~\ref{lemma:GS-vs-NNN}(i). Thus~\mbox{$G(S_0) \cong (\ZZ/2\ZZ)^4$}.

For every $1\le i\le 5$, consider the $G(S_0)$-torsor $\Lambda(S_i)$
such that $\Lambda(S_i)_{\mathbb{C}} \cong \Lambda(S_0)_{\mathbb{C}}$ and
the action of the non-trivial element of the Galois group
$\Gal(\mathbb{C}/\mathbb{R})$ on $\Lambda(S_i)_{\mathbb{C}}$
coincides with the action of the involution
${\iota_{klmn} \in \NNN \subset \WD}$,
where~\mbox{$\{i, k, l, m, n\} = \{1, 2, 3, 4, 5\}$}.
Let~$S_i$ be the twist of $S_0$ by the $G(S_0)$-torsor $\Lambda(S_i)$.
Note that there exists an isomorphism~\mbox{$\Delta(S_0) \cong \Delta(S_i)$}
of subschemes of $\PP^1$ by
Theorem~\ref{theorem:Skorobogatov} and Lemma~\ref{lemma:Delta-vs-Delta-triv}(iii).
Moreover, such isomorphism is unique (and even the isomorphism
$\Delta(S_0)_{\mathbb{C}} \cong \Delta(S_i)_{\mathbb{C}}$ is unique) by the choice of
the points $p_1,\ldots,p_5$.
Furthermore, the conjugation by $(ij) \in \SSS_5 \subset \WD$ provides an isomorphism of schemes~$\Lambda(S_i)$ and $\Lambda(S_j)$
preserving the natural incidence relation on~$\Lambda(S)_{\mathbb{C}}$,
because
$$
(ij) \iota_{jlmn} (ij) = \iota_{ilmn}
$$
in $\WD$.

Let us show that the surfaces $S_i$ and $S_j$ are not isomorphic for $i\neq j$.
Recall that there is a unique isomorphism between $\Delta(S_i)$ and $\Delta(S_j)$ as subschemes of $\mathbb{P}^1$, which induces isomorphisms
$$
G(S_i)\cong G(S_0)\cong G(S_j).
$$
Suppose that there is a $G(S_0)$-equivariant isomorphism between $\Lambda(S_i)_{\mathbb{C}}$ and $\Lambda(S_j)_{\mathbb{C}}$. Let~\mbox{$\gamma \in \Gal(\mathbb{C}/\mathbb{R})$} be the non-trivial element. Then the action of $\gamma$ on $\Lambda(S_i)_{\mathbb{C}}$ coincides with the action of $\iota_{jlmn} \in G(S_0)$, and the action of $\gamma$ on $\Lambda(S_j)_{\mathbb{C}}$ coincides with the action of $\iota_{ilmn} \in G(S_0)$, where~\mbox{$\{i, j, l, m, n\} = \{1, 2, 3, 4, 5\}$}. Therefore, there are no $G(S_0)$-equivariant isomorphisms between $\Lambda(S_i)_{\mathbb{C}}$ and $\Lambda(S_j)_{\mathbb{C}}$ which are at the same time $\Gal(\mathbb{C}/\mathbb{R})$-equivariant. Thus $\Lambda(S_i)$ and $\Lambda(S_j)$ are not isomorphic as torsors over the group scheme $G(S_i)\cong G(S_j)$. Hence the surfaces $S_i$ and $S_j$ are not isomorphic for~\mbox{$i \neq j$}
(this can be seen directly from the procedure which recovers the torsors $\Lambda(S_i)$ and $\Lambda(S_j)$ out of the surfaces $S_i$ and $S_j$,
or from Theorem~\ref{theorem:Skorobogatov}).
\end{example}

\section{Birational models}
\label{section:birational}

In this section we study birational models of minimal del Pezzo surfaces
of degree~$4$ and prove Theorem~\ref{theorem:dP4-unique}.

\begin{lemma}\label{lemma:discriminant}
Let $S$ be a del Pezzo surface of degree $4$ over
a field $\KK$. Let $X$ be a smooth cubic surface
obtained as a blow up $f\colon X\to S$ at a $\KK$-point $p$.
Let $\phi\colon X\to\PP^1$ be the conic bundle defined by the $f$-exceptional line on $X$, and let $\Delta(\phi)\subset\PP^1$ be the
discriminant of~$\phi$, i.e. the closed subscheme of $\PP^1$ parameterizing the degenerate fibers of~$\phi$.
Then there is a natural isomorphism $\cP\stackrel{\sim}\longrightarrow\PP^1$ which restricts
to an isomorphism~\mbox{$\Delta(S)\stackrel{\sim}\longrightarrow\Delta(\phi)$}.
\end{lemma}

\begin{proof}
The fibers of the conic bundle $\phi_{\KK^{sep}}$ are proper transforms of the hyperplane sections of~$S_{\KK^{sep}}$ that are singular at~$p$.
In other words, a fiber of $\phi_{\KK^{sep}}$ is the proper transform of a curve
cut out on $S_{\KK^{sep}}$ by a hyperplane $H\subset\PP^4$ which passes through the tangent plane~$\Pi$ to $S_{\KK^{sep}}$ at~$p$.

Let $R\in\cP_{\KK^{sep}}$ be a quadric passing through $S_{\KK^{sep}}$, and let $H_R$ be the tangent hyperplane to~$R$ at $p$.
Associating to $R$ the image under $\phi_{\KK^{sep}}$ of the
proper transform of $H_R\cap S$ gives a $\Gal(\KK^{sep}/\KK)$-equivariant isomorphism
$$
\delta_{\KK^{sep}}\colon \cP_{\KK^{sep}}\stackrel{\sim}\longrightarrow\PP^1.
$$
Thus, $\delta_{\KK^{sep}}$ can be obtained by extension of scalars from an isomorphism
$\delta\colon \cP\stackrel{\sim}\longrightarrow\PP^1$. Furthermore, if $R$ is singular, then $H_R$ contains two planes $\Pi_1,\Pi_2\subset R$
passing through~$p$. Therefore, the curve $H_R\cap S$ is reducible in this case, which means that
$$
\delta(\Delta(S))\subset\Delta(\phi).
$$
On the other hand, we know that $\Delta(S)_{\KK^{sep}}$ and $\Delta(\phi)_{\KK^{sep}}$ are finite schemes that consist of the same number of
points. Hence $\delta$ restricts to an isomorphism between $\Delta(S)$ and~$\Delta(\phi)$.
\end{proof}

Recall the following standard terminology.

\begin{definition}\label{definition:elementary-transformation}
Let $X$ be a smooth surface, and let $\phi\colon X\to C$ be a conic bundle.
A \emph{transformation of the conic bundle $\phi$}
is a birational map $\theta\colon X\dasharrow X'$ to another conic bundle $\phi'\colon X\to C'$
which fits into a commutative diagram
$$
\xymatrix{
X\ar@{->}[d]_{\phi}\ar@{-->}[rr]^{\theta}&& X'\ar@{->}[d]^{\phi'}\\
C\ar@{->}[rr]^{\sim}&& C
}
$$
and induces an isomorphism in a neighborhood of degenerate fibers of $\phi$ and $\phi'$.
\end{definition}

\begin{lemma}\label{lemma:markings-wd-equivariant-cb}
Let $X$ be a smooth cubic surface over a field $\KK$, and let $\phi\colon X\to \PP^1$ be a conic bundle.
Consider a transformation $\theta\colon X \to X'$ of $\phi$, and denote by $\phi'\colon X'\to \PP^1$ the resulting conic bundle.
Assume that $X'$ is again a (smooth) cubic surface.
Consider a marking
$$
E_1\cup F_1,\ldots,E_5\cup F_5
$$
of $\phi$. Since the degenerate fibers (and their irreducible components)
of~$\phi$ and~$\phi'$ are naturally identified with each other,
this provides a marking
$$
E'_1\cup F'_1,\ldots,E'_5\cup F'_5
$$
of $\phi'$, where $E'_i = \theta(E_i)$ and $F'_i = \theta(F_i)$. These markings give actions of $\WD$ on~$\Pic(X_{\KK^{sep}})$ and~$\Pic(X'_{\KK^{sep}})$.
Then there exists a $\WD$-equivariant isomorphism
$$
\tilde{\theta} \colon \Pic(X_{\KK^{sep}}) \stackrel{\sim}\longrightarrow \Pic(X'_{\KK^{sep}}),
$$
such that either $\tilde{\theta}(E_i) = E'_i$ for any $i$, or $\tilde{\theta}(E_i) = F'_i$ for any $i$.

Moreover, let the group $G$ be either
$\Gal(\KK^{sep}/\KK)$, or a subgroup of $\Aut(X)$. Suppose that~$\theta$ is $G$-equivariant;
thus, the group $G$ acts on $X'$ as well. The action of $G$ on~\mbox{$\Pic(X_{\KK^{sep}})$} and~\mbox{$\Pic(X'_{\KK^{sep}})$} gives homomorphisms~\mbox{$\pi_X \colon G \to \WD$} and~\mbox{$\pi_{X'} \colon G \to \WD$}. Then the following diagram is commutative:
$$
\xymatrix{
&G\ar@{->}[ld]_{\pi_X}\ar@{->}[rd]^{\pi_{X'}}&\\
\WD\ar@{->}[rr]^{\sim}&& \WD
}
$$
\end{lemma}

\begin{proof}
After replacing the original marking
$E_1\cup F_1,\ldots,E_5\cup F_5$ by the equivalent marking~\mbox{$F_1\cup E_1,\ldots,F_5\cup E_5$}
if necessary, we may assume that
every $(-1)$-curve on $X_{\KK^{sep}}$ that is a section of~$\phi_{\KK^{sep}}$ intersects an odd number of curves among~\mbox{$E_1,\ldots,E_5$}, see
Corollary~\ref{corollary:odd-number-of-fibers};
recall from Remark~\ref{remark: equivalent-markings-gives-same-wd-actions}
that equivalent markings of~$\phi$ give the same action of~$\WD$ on~$\Pic(X_{\KK^{sep}})$.
By Lemma~\ref{lemma:unique-section-odd} there exists a unique section $C$ of $\phi_{\KK^{sep}}$ with self-intersection~$-1$
that intersects all the curves $E_i$. In the same way we can find a section~$C'$ of $\phi'_{\KK^{sep}}$
with self-intersection~$-1$ that intersects all the curves~$E'_i$.

The group $\Pic(X_{\KK^{sep}})$ is generated by the classes of $C$, $E_i$, and $F_i$ with relations~\mbox{$E_i + F_i = F$}, where $F$ is the class of a fiber of $\phi_{\KK^{sep}}$.
Similarly, the group $\Pic(X'_{\KK^{sep}})$ is generated by the classes of $C'$, $E'_i$, and $F'_i$ with relations $E'_i + F'_i = F'$, where $F'$ is the class of a fiber of~$\phi'_{\KK^{sep}}$. One has $C \cdot E_i = C' \cdot E'_i = 1$.
Consider the map
$$
\tilde{\theta}\colon \Pic(X_{\KK^{sep}}) \to \Pic(X'_{\KK^{sep}})
$$
given by $\tilde{\theta}(C) = C'$, $\tilde{\theta}(E_i) = E'_i$, and $\tilde{\theta}(F_i) = F'_i$. This map preserves the intersection form. Let us show that $\tilde{\theta}$ is equivariant with respect to the action
of~\mbox{$\WD = \NNN \rtimes \HHH$} on~\mbox{$\Pic(X_{\KK^{sep}})$} and~\mbox{$\Pic(X'_{\KK^{sep}})$}
given by the markings. For $\sigma \in \HHH$ one has
$$
\tilde{\theta}(\sigma E_i) = \tilde{\theta}(E_{\sigma(i)}) = E'_{\sigma(i)} = \sigma E'_i = \sigma \tilde{\theta}(E_i).
$$
Note that for any $i\in\{1,\ldots,5\}$ and $\iota_{i_1\ldots i_l} \in \NNN$ one has
$$
\iota_{i_1\ldots i_l} E_i = E_i,\quad \iota_{i_1\ldots i_l} F_i = F_i,\quad \iota_{i_1\ldots i_l} E'_i = E'_i,\quad \iota_{i_1\ldots i_l} F'_i = F'_i,
$$
if $i \notin \{i_1, \ldots, i_l\}$, and
$$
\iota_{i_1\ldots i_l} E_i = F_i,\quad \iota_{i_1\ldots i_l} F_i = E_i,\quad \iota_{i_1\ldots i_l} E'_i = F'_i,\quad \iota_{i_1\ldots i_l} F'_i = E'_i.
$$
if $i \in \{i_1, \ldots, i_l\}$. Therefore
$$
\tilde{\theta}(\iota E_i) = \iota E'_i = \iota \tilde{\theta}(E_i)
$$
for any $\iota\in \NNN$.
The action of $\WD$ and the map $\tilde{\theta}$ preserve the intersection form. Therefore
$$
\tilde{\theta}(\iota\sigma C) \cdot \tilde{\theta}(\iota\sigma E_i) = C \cdot E_i = \iota\sigma \tilde{\theta}(C) \cdot \iota\sigma \tilde{\theta}(E_i) = \iota\sigma \tilde{\theta}(C) \cdot \tilde{\theta}(\iota\sigma E_i).
$$
But there is a unique class $D$ in $\Pic(X'_{\KK^{sep}})$ such that $D \cdot \tilde{\theta}(\iota\sigma E_i) = 1$, $D \cdot F' = 1$, and~\mbox{$D^2 = -1$}. Thus
$$
\tilde{\theta}(\iota\sigma C) = \iota\sigma \tilde{\theta}(C),
$$
and so the map $\tilde{\theta}$ is $\WD$-equivariant.

Now consider the action of $G$ on $\Pic(X_{\KK^{sep}})$ and $\Pic(X'_{\KK^{sep}})$.
Due to our choice of the markings of $\phi$ and $\phi'$ which induce this action,
for any $i\in\{1,\ldots,5\}$ and~\mbox{$g \in G$} there exists $j$
such that either
$$
g E_i = E_j,\quad g F_i = F_j,\quad g E'_i = E'_j,\quad g F'_i = F'_j,
$$
or
$$
g E_i = F_j,\quad g F_i = E_j,\quad g E'_i = F'_j,\quad g F'_i = E'_j.
$$
Therefore
$$
\tilde{\theta}(g E_i) = g E'_i = g \tilde{\theta}(E_i).
$$
As in the case of the action of $\WD$, one can show that $\tilde{\theta}(g C) = g \tilde{\theta}(C)$, and the map~$\tilde{\theta}$ is $G$-equivariant. It means that for any $g \in G$ the action of $\pi_X(g)$ on $\Pic(X'_{\KK^{sep}})$ induced by the isomorphism
$$
\tilde{\theta} \colon \Pic(X_{\KK^{sep}}) \stackrel{\sim}\longrightarrow \Pic(X'_{\KK^{sep}})
$$
coincides with the action of $\pi_{X'}(g)$. The isomorphism $\tilde{\theta}$ is $\WD$-equivariant, therefore~\mbox{$\pi_X(g) = \pi_{X'}(g)$}, and the second assertion of lemma holds.
\end{proof}

\begin{remark}
\label{remark:theta-maps-E-F}
In the notation of Lemma~\ref{lemma:markings-wd-equivariant-cb},
we have proved that either $\tilde{\theta}(E_i) = E'_i$ for any~$i$, or $\tilde{\theta}(E_i) = F'_i$ for any~$i$.
The former option is realized for the isomorphism $\tilde{\theta}$ constructed in the proof of the lemma
after a choice of one of the two equivalent markings for each of the conic bundles~$\phi$
and~$\phi'$. The latter option is realized if we had to replace the original marking
by the equivalent one for $\phi$ but not for $\phi'$, or vice versa.
\end{remark}

The isomorphism $\tilde{\theta}$ constructed in Lemma~\ref{lemma:markings-wd-equivariant-cb} is the
unique $\WD$-equivariant isomorphism preserving the intersection form
due to the following observation (which we will not use in the sequel, but still
find it rather instructive).

\begin{proposition}\label{proposition:tilde-theta-unique}
In the notation of Lemma~\ref{lemma:markings-wd-equivariant-cb}, let
$$
\tilde{\vartheta}\colon \Pic(X_{\KK^{sep}}) \stackrel{\sim}\longrightarrow \Pic(X'_{\KK^{sep}})
$$
be a $\WD$-equivariant isomorphism which preserves the intersection form.
Then~\mbox{$\tilde{\vartheta}=\tilde{\theta}$}.
\end{proposition}

\begin{proof}
We are going to show that the isomorphism $\vartheta$ with the required properties is unique.
To start with, note that $\tilde{\vartheta}(E_1)$ is either $E'_1$, or $F'_1$ since otherwise $\tilde{\vartheta}$ is not $\HHH$-equivariant.
Therefore for any $i$ one can consider $\sigma = (1i) \in \HHH$ and obtain that either
$$
\tilde{\vartheta}(E_i) = \tilde{\vartheta}(\sigma E_1) = \sigma\tilde{\vartheta}(E_1) = \sigma E'_1 = E'_i,
$$
or
$$
\tilde{\vartheta}(E_i) = \tilde{\vartheta}(\sigma E_1) = \sigma \tilde{\vartheta}(E_1) = \sigma F'_1 = F'_i.
$$
One has $\tilde{\vartheta}(-K_{X_{\KK^{sep}}}) = -K_{X'_{\KK^{sep}}}$, since for $\tilde{\vartheta}(-K_{X_{\KK^{sep}}})$ the following equalities hold
$$
\tilde{\vartheta}(-K_{X_{\KK^{sep}}}) \cdot E'_i = \tilde{\vartheta}(-K_{X_{\KK^{sep}}}) \cdot F'_i = 1, \qquad \left(\tilde{\vartheta}(-K_{X_{\KK^{sep}}})\right)^2 = 4,
$$
and the unique element in $\Pic(X'_{\KK^{sep}})$ with such properties is $-K_{X'_{\KK^{sep}}}$. Moreover, we have~\mbox{$\tilde{\vartheta}(F) = F'$},
where $F$ and $F'$ are classes of fibers of $\phi$ and $\phi'$,  respectively.
If~\mbox{$\tilde{\vartheta}(E_i) = E'_i$} for any $i$,
then this uniquely defines $\tilde{\vartheta}$, since the classes $-K_{X_{\KK^{sep}}}$, $F$,
and~$E_i$ form a basis in the vector space  $\Pic(X_{\KK^{sep}}) \otimes \QQ$.
The same applies to the case when~\mbox{$\tilde{\vartheta}(E_i) = F'_i$} for any~$i$.

On the other hand, there cannot simultaneously exist two isomorphisms
$\tilde{\vartheta}_E$ and $\tilde{\vartheta}_F$ such that
$\tilde{\vartheta}_E(E_i)=E'_i$ and $\tilde{\vartheta}_F(E_i)=F'_i$ for any $i$.
Indeed, consider the automorphism
$$
\chi = \tilde{\vartheta}^{-1}_F \tilde{\vartheta}_E
$$
of $\Pic(X_{\KK^{sep}})$. One has $\chi(E_i) = F_i$ and $\chi(F_i) = E_i$. For a class $D$ of a $(-1)$-curve on~$X_{\KK^{sep}}$ that is a section of $\phi_{\KK^{sep}}$, one has
$$
D \cdot E_i + \chi(D) \cdot E_i = D \cdot E_i + D \cdot F_i = 1.
$$
Therefore the numbers of curves among $E_1,\ldots,E_5$ that intersect the curves corresponding to $D$ and $\chi(D)$ have different parity, and we obtain a contradiction with Corollary~\ref{corollary:parity-well-defined}.
\end{proof}

\medskip
The next result was obtained in \cite[Proposition~2.1]{Iskovskikh-positive} in the case of perfect fields, and generalized to the case of arbitrary fields in~\cite{BFSZ};
see also Proposition~\ref{proposition:Iskovskikh-elementary-transformation-equivariant}
below for a proof.

\begin{proposition}[{\cite[Proposition 4.32(3)]{BFSZ}}]
\label{proposition:Iskovskikh-elementary-transformation}
Let $\phi\colon X\to \PP^1$ be a conic bundle.
Suppose that~{$\rkPic(X)=2$} and $K_X^2=3$.
Then $X$ is a cubic surface containing a line.
\end{proposition}

\begin{corollary}
\label{corollary:Iskovskikh-elementary-transformation}
Let $X$ be a smooth cubic surface such that~{$\rkPic(X)=2$} and $X$ contains a line.
Let $\phi\colon X\to\PP^1$ be the conic bundle, and let~{$\theta\colon X\dashrightarrow X'$}
be a transformation of~$\phi$.
Then $X'$ is a cubic surface such that~{$\rkPic(X')=2$} and $X'$ contains a line.
\end{corollary}

\begin{proof}
A transformation of a conic bundle
does not change the Picard rank and the canonical degree.
Thus, the required assertion follows from Proposition~\ref{proposition:Iskovskikh-elementary-transformation}.
\end{proof}

Corollary~\ref{corollary:Iskovskikh-elementary-transformation} tells us that if $X$ is smooth cubic surface over a field $\KK$ obtained as a blow up of a del Pezzo surface $S$ of degree $4$ with $\rkPic(S)=1$ at a $\KK$-point, and~\mbox{$\theta\colon X\dasharrow X'$}
is a transformation of the conic bundle $X\to\PP^1$, then $X'$ can be also obtained as a blow up of some
del Pezzo surface $S'$ of degree $4$ with~\mbox{$\rkPic(S')=1$} at a $\KK$-point. Thus, in this case we have a diagram
\begin{equation}\label{eq:elementary-transformation}
\xymatrix{
&X\ar@{->}[d]_{\phi}\ar@{-->}[rr]^{\theta}\ar@{->}[ld]_{f}&& X'\ar@{->}[d]^{\phi'}\ar@{->}[rd]^{f'}&\\
S&\PP^1\ar@{->}[rr]^{\sim}&&\PP^1& S'
}
\end{equation}

The next result provides additional information about the surface~$S'$ in the diagram~\eqref{eq:elementary-transformation}.

\begin{proposition}\label{proposition:Lambda-Lambda-prime-WD-equivariant}
Let $S$ be a del Pezzo surface of degree $4$ over
a field $\KK$. Let~$X$ be a smooth cubic surface
obtained as a blow up $f\colon X\to S$ of $S$ at a $\KK$-point, and let~\mbox{$\phi\colon X\to\PP^1$} be the conic bundle defined by the $f$-exceptional line on $X$. Let
$\theta\colon X\dasharrow X'$ be a transformation of~$\phi$.
Suppose that there exists a contraction $f'\colon X'\to S'$
to a del Pezzo surface~$S'$ of degree~$4$.
Then there exist markings $L_1,\ldots,L_5$ and $L_1',\ldots,L_5'$ of $S$ and $S'$, and an isomorphism
$$
\nu\colon\Lambda(S)\stackrel{\sim}\longrightarrow\Lambda(S')
$$
of the Hilbert schemes of lines on $S$ and~$S'$ such that $\nu_{\KK^{sep}}$ sends
the line $L_i$ to $L_i'$ for~\mbox{$i\in\{1,\ldots,5\}$}, and
the diagram
$$
\xymatrix{
\WD\times\Lambda(S)\ar@{->}^{\mathrm{id}\times\nu}[rr]\ar@{->}[d]&&
\WD\times\Lambda(S')\ar@{->}[d]\\
\Lambda(S)\ar@{->}^{\nu}[rr]&&\Lambda(S')
}
$$
where the vertical arrows correspond to the action of $\WD$ on $\Lambda(S)$ and $\Lambda(S')$ induced by the markings,
is commutative.
In other words, the isomorphism $\nu$ is
equivariant with respect to the action of $\WD$ on
$\Lambda(S)$ and $\Lambda(S')$.
\end{proposition}

\begin{proof}
Consider any marking $L_1,\ldots,L_5$ of $S$. Applying Construction~\ref{construction:markings-1-to-1-S-to-X}, one obtains a marking
$$
E_1\cup F_1,\ldots,E_5\cup F_5
$$
of $\phi$.
Since the degenerate fibers (and their irreducible components)
of~$\phi$ and~$\phi'$ are naturally identified with each other,
this provides a marking of $\phi'$. Finally, applying Construction~\ref{construction:markings-1-to-1-X-to-S}, we obtain a marking $L_1',\ldots,L_5'$ of $S'$. Each of these markings gives\ an action of $\WD$ on $\Pic(S_{\KK^{sep}})$, $\Pic(X_{\KK^{sep}})$, $\Pic(X'_{\KK^{sep}})$, and $\Pic(S'_{\KK^{sep}})$, respectively. The natural embeddings $\Pic(S_{\KK^{sep}}) \hookrightarrow \Pic(X_{\KK^{sep}})$ and $\Pic(S'_{\KK^{sep}}) \hookrightarrow \Pic(X'_{\KK^{sep}})$ are $\WD$-equivariant by Lemma~\ref{lemma:markings-wd-equivariant}.

Consider a $\WD$-equivariant isomorphism
$$
\tilde{\theta} \colon \Pic(X_{\KK^{sep}}) \stackrel{\sim}\longrightarrow \Pic(X'_{\KK^{sep}})
$$
constructed in Lemma~\ref{lemma:markings-wd-equivariant-cb}.
One has either $\tilde{\theta}(E_i) = E'_i$ for any $i$, or $\tilde{\theta}(E_i) = F'_i$ for any~$i$.
We see from Construction~\ref{construction:markings-1-to-1-S-to-X} (or from
Construction~\ref{construction:markings-1-to-1-X-to-S}) that for a given index~$k$
we have
\begin{equation}\label{eq:kk-vs-kj}
f^{-1}_*L_k \cdot E_k \neq f^{-1}_*L_k \cdot E_j
\end{equation}
for any $j \neq k$.
More precisely, for one of the two equivalent markings of $\phi$ the left hand side
of~\eqref{eq:kk-vs-kj} equals~$1$, while the right hand side equals~$0$ for any $j\neq k$,
and vice versa for the other marking.
Therefore
$$
\tilde{\theta} (f^{-1}_*L_k) \cdot E'_k \neq \tilde{\theta} (f^{-1}_*L_k) \cdot E'_j
$$
for any $j \neq k$.
Also, $\tilde{\theta} (f^{-1}_*L_k)$ intersects a fiber of $\phi'$ by $1$
(because $f^{-1}_*L_k$ intersects a fiber of $\phi$ by $1$), and we have
$$
\big(\tilde{\theta} (f^{-1}_*L_k)\big)^2=-1.
$$
Applying Construction~\ref{construction:markings-1-to-1-X-to-S},
one concludes that
$$
f'\big(\tilde{\theta} (f^{-1}_*L_k)\big) = L'_k.
$$
Note that $\tilde{\theta}$ preserves the intersection form, therefore it maps any class of a $(-1)$-curve in~\mbox{$\Pic(X_{\KK^{sep}})$} to a class of a $(-1)$-curve in $\Pic(X'_{\KK^{sep}})$. Moreover, if a $(-1)$-curve is a section of~$\phi_{\KK^{sep}}$, then the image of its class is the class of a section of~$\phi'_{\KK^{sep}}$, again with self-intersection equal to~$-1$.
The $16$ sections of $\phi_{\KK^{sep}}$ with self-intersection $-1$ are mapped by $f$ to $16$ lines on $S_{\KK^{sep}}$, and the $16$ sections of $\phi'_{\KK^{sep}}$ with self-intersection $-1$ are mapped by $f'$ to $16$ lines on $S'_{\KK^{sep}}$, see Remark~\ref{remark:lines-on-cubic}.
Therefore the restriction of $\tilde{\theta}$ to the set of sections of $\phi_{\KK^{sep}}$ with self-intersection $-1$ gives an isomorphism
$$
\nu_{\KK^{sep}} \colon \Lambda(S_{\KK^{sep}})\stackrel{\sim}\longrightarrow\Lambda(S'_{\KK^{sep}})
$$
such that $\nu_{\KK^{sep}}$ sends the line $L_i$ to $L_i'$, since $\tilde{\theta} (f^{-1}_*L_k) = f'^{-1}_*L'_k$.

The action of the group $\Gal(\KK^{sep}/\KK)$ on the Picard groups $\Pic(S_{\KK^{sep}})$, $\Pic(X_{\KK^{sep}})$, $\Pic(X'_{\KK^{sep}})$, and~\mbox{$\Pic(S'_{\KK^{sep}})$} gives homomorphisms $\pi_S$, $\pi_X$, $\pi_{X'}$, and $\pi_{S'}$ from $\Gal(\KK^{sep}/\KK)$ to~$\WD$. All these homomorphisms coincide
by Lemmas~\ref{lemma:markings-wd-equivariant} and~\ref{lemma:markings-wd-equivariant-cb}.
Therefore the constructed $\WD$-equivariant isomorphism~$\nu_{\KK^{sep}}$ is $\Gal(\KK^{sep}/\KK)$-equivariant,
which means that there exists a $\WD$-equivariant isomorphism $\nu \colon \Lambda(S) \to \Lambda(S')$
such that $\nu_{\KK^{sep}}$ is obtained by extension of scalars from~$\nu$.
\end{proof}

\begin{corollary}\label{corollary:dP4-isomorphic}
Let $S$ be a del Pezzo surface of degree $4$ over
a field $\KK$. Let~$X$ be a smooth cubic surface
obtained as a blow up $f\colon X\to S$ at a $\KK$-point, and let~\mbox{$\phi\colon X\to\PP^1$}
be the conic bundle defined by the $f$-exceptional line on $X$. Let
$\theta\colon X\dasharrow X'$ be a transformation of~$\phi$.
Suppose that there exists a contraction $f'\colon X'\to S'$
to a del Pezzo surface~$S'$ of degree~$4$.
Then
\begin{itemize}
\item[(i)] $\Delta(S)$ is naturally isomorphic to $\Delta(S')$ as
closed subschemes of $\PP^1$;

\item[(ii)] $G(S)$ is naturally isomorphic to $G(S')$ as group schemes;

\item[(iii)] the $\WD$-equivariant isomorphism $\nu\colon\Lambda(S)\stackrel{\sim}\longrightarrow\Lambda(S')$ from Proposition~\ref{proposition:Lambda-Lambda-prime-WD-equivariant} is isomorphism of torsors over $G(S)\cong G(S')$;

\item[(iv)] there exists an isomorphism $\psi\colon S \stackrel{\sim}\longrightarrow S'$ such that the induced map
$$
\tilde{\psi}\colon\Lambda(S)\stackrel{\sim}\longrightarrow\Lambda(S')
$$
coincides with $\nu$.
\end{itemize}
\end{corollary}

\begin{proof}
Assertion~(i) immediately follows from Lemma~\ref{lemma:discriminant}
and the fact that an elementary transformation of conic bundles does not affect the discriminant.
Assertion~(ii) follows from the construction of the group scheme $G(S)$, which depends only on
the isomorphism class of~$\Delta(S)$.

The $\WD$-equivariant isomorphism
$$
\nu\colon\Lambda(S)\to \Lambda(S')
$$
is equivariant with respect to the subgroup $\NNN\subset\WD$.
On the other hand, by Lemma~\ref{lemma:GS-vs-NNN}(i)
there exist $\Gal(\KK^{sep}/\KK)$-equivariant isomorphisms
$$
G(S)_{\KK^{sep}}\cong\NNN\cong G(S')_{\KK^{sep}},
$$
where $\NNN\cong (\ZZ/2\ZZ)^4$ is a normal subgroup of $\WD$.
This gives a $\Gal(\KK^{sep}/\KK)$-equivariant isomorphism
$$
\gamma_{\KK^{sep}}\colon G(S)_{\KK^{sep}}\to G(S')_{\KK^{sep}}
$$
which fits into a $\Gal(\KK^{sep}/\KK)$-equivariant commutative diagram
$$
\xymatrix{
G(S)_{\KK^{sep}}\times\Lambda(S)_{\KK^{sep}}\ar@{->}^{\gamma_{\KK^{sep}}\times\nu_{\KK^{sep}}}[rr]\ar@{->}[d]&&
G(S')_{\KK^{sep}}\times\Lambda(S')_{\KK^{sep}}\ar@{->}[d]\\
\Lambda(S)_{\KK^{sep}}\ar@{->}^{\nu_{\KK^{sep}}}[rr]&&\Lambda(S')_{\KK^{sep}}
}
$$
Hence, $\nu$ is an isomorphism of torsors over
$G(S)\cong G(S')$.
This proves assertion~(iii).

Finally, by Corollary~\ref{corollary:Skorobogatov}, the isomorphisms from assertions (i) and (iii) provide an isomorphism $\psi\colon S \stackrel{\sim}\longrightarrow S'$ such that the induced isomorphism $\tilde{\psi}\colon\Lambda(S)\stackrel{\sim}\longrightarrow\Lambda(S')$ coincides with~$\nu$.
\end{proof}

The next result is just a more precise version of Theorem~\ref{theorem:Iskovskikh-old}(ii).
It was initially proved over perfect fields by V.\,A.\,Iskovskikh,
see~\mbox{\cite[Theorems 2.5 and 2.6]{Isk96}}, and cf.~\mbox{\cite[Theorem 3b]{Iskovskikh-positive}}.
The proof in the case of arbitrary field is based on the results of~\cite{BFSZ}.
We refer the reader to Definitions~\ref{definition:Sarkisov-general} and~\ref{definition:birational-GB} for the terminology
we are going to use.

\begin{theorem}
\label{theorem:Iskovskikh-dP4-models}
Let $S$ be a del Pezzo surface of degree~$4$ over
a  field $\KK$
such that~\mbox{$\rkPic(S)=1$}.
Then all the Mori fiber spaces birational to $S$ are either del Pezzo
surfaces of Picard rank~$1$ and degree~$4$, or smooth cubic surfaces with a structure
of a relatively minimal conic bundle over~$\PP^1$.
Furthermore, all birational maps from $S$ to Mori fiber spaces are
obtained by a composition of maps of the following types:
\begin{itemize}
\item automorphisms;

\item birational Geiser involutions of del Pezzo surfaces of degree $4$ with Picard rank~$1$ defined by points of degree~$2$ in Sarkisov general
position;

\item birational Bertini involutions of del Pezzo surfaces of degree $4$ with Picard rank~$1$ defined by points of degree~$3$ in Sarkisov general
position;

\item blow ups of $\KK$-points in Sarkisov general position
on del Pezzo surfaces of degree $4$ with Picard rank~$1$,
and their inverses;

\item transformations between conic bundles on smooth cubic surfaces of
Picard rank~$2$.
\end{itemize}
\end{theorem}

\begin{proof}
By \cite[Theorem A]{BFSZ}, any birational map between Mori fiber spaces of dimension $2$ is a composition of Sarkisov links and automorphisms of Mori fiber spaces.
On the other hand, from Proposition~\ref{proposition:links-starting-from-dP4} we know the description of Sarkisov links starting from~$S$: it is either
a link of type $\mathrm{II}$ which is a birational Geiser or Bertini involution~\mbox{$S\dasharrow S$}, or
a link of type $\mathrm{I}$ which is a blow up of a $\KK$-point resulting in a smooth cubic surface with a structure of a relatively minimal conic bundle.
Any Sarkisov link from the latter Mori fiber space is either the inverse to the link of the previous type (so that it again brings us to a del Pezzo surface
of degree~$4$ and Picard rank~$1$), or a link of type~$\mathrm{II}$, see~\mbox{\cite[Proposition 5.6]{BFSZ}}.
In the latter case, the link is a transformation between conic bundles on smooth cubic surfaces by Corollary~\ref{corollary:Iskovskikh-elementary-transformation}.
Therefore, the required assertion follows by induction on the number of links in the decomposition.
\end{proof}

Theorem~\ref{theorem:Iskovskikh-dP4-models} immediately implies Theorem~\ref{theorem:Iskovskikh-old}.
Moreover, one can deduce the following more precise assertion.

\begin{corollary}
\label{corollary:Iskovskikh-dP4-models}
Let $S$ be a del Pezzo surface of degree $4$ over
a field $\KK$ such that~\mbox{$\rkPic(S)=1$}.
Then all the Mori fiber spaces birational to $S$ are~$S$ itself
and smooth cubic surfaces obtained by a blow up
of a $\KK$-point on $S$. In particular, if $S(\KK)=\varnothing$, then the only Mori fiber space birational to $S$ is $S$ itself.
\end{corollary}

\begin{proof}
Let $\chi \colon S \dashrightarrow S'$ be a birational map such that $S'$ is a del Pezzo surface of Picard rank $1$ and degree~$4$.
By Theorem~\ref{theorem:Iskovskikh-dP4-models},
the map $\chi$ can be decomposed into a sequence of birational maps
$$
S = S_0 \stackrel{\chi_0}\dashrightarrow S_1 \stackrel{\chi_1}\dashrightarrow \ldots \stackrel{\chi_{n-1}}\dashrightarrow S_n = S',
$$
where each $S_i$ is either del Pezzo surface of Picard rank $1$ and degree~$4$, or a smooth cubic surfaces with a structure
of a relatively minimal conic bundle over~$\PP^1$, and each $\chi_i$ is a birational map of one of the types listed in Theorem~\ref{theorem:Iskovskikh-dP4-models}.
Let $i_0, \ldots, i_k$ be integers such that
$$
0 = i_0 < \ldots < i_k = n,
$$
any surface $S_{i_j}$ is a del Pezzo surface of degree $4$, and $S_l$ is a cubic surface for
$$
l \notin \{i_0, \ldots, i_k\}.
$$

Let us show that all surfaces $S_{i_j}$ are isomorphic.
Consider any non-negative integer~\mbox{$j < k$}. If
$i_{j+1} = i_j + 1$,
then the map $\chi_{i_j}$ is either an automorphism, or a birational Geiser or Bertini involution. Therefore $S_{i_j} \cong S_{i_{j+1}}$.

If $i_{j+1} > i_j + 1$, then
$$
\chi_{i_j}^{-1} \colon S_{i_j+1} \rightarrow S_{i_j}
$$
and
$$
\chi_{i_{j+1}-1} \colon S_{i_{j+1} - 1} \rightarrow S_{i_{j+1}}
$$
are blow ups of $\KK$-points
on del Pezzo surfaces of degree $4$ and
$$
\chi_{i_{j+1}-2}\ldots\chi_{i_j+1} \colon S_{i_j+1} \dashrightarrow S_{i_{j+1} - 1}
$$
is a transformation between conic bundles on smooth cubic surfaces. Therefore $S_{i_j} \cong S_{i_{j+1}}$ by Corollary \ref{corollary:dP4-isomorphic}(iv). Hence all surfaces $S_{i_j}$ are isomorphic.

In particular, we see that any del Pezzo surface~$S'$ of Picard rank~$1$ and degree~$4$ which is birational to~$S$ is actually isomorphic to $S$.
On the other hand, if~$X'$ is a smooth cubic surface with a structure of a relatively minimal conic bundle which is birational to~$S$, then by Corollary~\ref{corollary:Iskovskikh-elementary-transformation} there is a line $E$ on $X'$. One can blow down this line and obtain a del Pezzo surface~$S'$ of Picard rank~$1$ and degree~$4$. We have already proved that~${S' \cong S}$. Therefore, the cubic surface~$X'$ is isomorphic to a blow up of a $\KK$-point on~$S$. On the other hand, we know from Theorem~\ref{theorem:Iskovskikh-dP4-models} that any Mori fiber space birational to $S$ is either a del Pezzo surface of Picard rank $1$ and degree~$4$, or a smooth cubic surfaces with a structure of a relatively minimal conic bundle. This proves the first of the required assertions.

If $S(\KK)=\varnothing$, then it is impossible to blow up a $\KK$-point on $S$. Therefore the only Mori fiber space birational to $S$ is $S$ itself.
\end{proof}

Now we can prove Theorem~\ref{theorem:dP4-unique}.

\begin{proof}[Proof of Theorem~\ref{theorem:dP4-unique}]
If $S(\KK)=\varnothing$, then the required assertion is given by
Corollary~\ref{corollary:Iskovskikh-dP4-models}
(or by Theorem~\ref{theorem:Iskovskikh-old}(i)).

Suppose that $S(\KK)\neq\varnothing$,
and let $S'$ be a del Pezzo surface of degree $4$ birational to $S$.
If $\rkPic(S')=1$, then $S'\cong S$ by Corollary~\ref{corollary:Iskovskikh-dP4-models}.
Otherwise $S'$ either has a structure of a relatively minimal conic bundle, or
is birational to a Mori fiber space $S''$ with $K_{S''}^2>4$.
This is also impossible by Corollary~\ref{corollary:Iskovskikh-dP4-models}
(or by Theorem~\ref{theorem:Iskovskikh-old}(ii)).
\end{proof}

\section{Equivariant birational models}
\label{section:G-birational}

In this section we discuss equivariant birational geometry of del Pezzo surfaces of degree~$4$.
Suppose that $\Gamma$ is a group, and $X_1$ and $X_2$ are two varieties acted on by~$\Gamma$, i.e. there are
homomorphisms $\rho_i\colon \Gamma\to\Aut(X_i)$. In this case we say that a rational map~\mbox{$\chi\colon X_1\dasharrow X_2$}
is $\Gamma$-equivariant, if for any element $g\in\Gamma$
one has
$$
\chi\rho_1(g)=\rho_2(g)\chi.
$$
If the homomorphisms $\rho_i$ are clear from the context, we will sometimes drop them from the notation and write $g$ instead of~$\rho_i(g)$.

Our current goal is to give a sketch of a proof of Theorem~\ref{theorem:dP4-unique-equivariant}.
To turn it into a rigorous proof, one needs to spell out the basics of $\Gamma$-Sarkisov theory for regular surfaces over arbitrary fields, using the same arguments as in~\cite{BFSZ}. Namely, we need the existence of decomposition of a $\Gamma$-equivariant birational map into $\Gamma$-Sarkisov links, and a classifictaion of such links.

To start with, we will need an equivariant version of
Proposition~\ref{proposition:Iskovskikh-elementary-transformation},
cf.~also~\mbox{\cite[Propositions 4.30(3) and 4.32(3)]{BFSZ}} and~\mbox{\cite[Proposition 5.2]{Pr18}}.

\begin{proposition}
\label{proposition:Iskovskikh-elementary-transformation-equivariant}
Let $\phi\colon X\to \PP^1$ be a conic bundle over a field~$\KK$
with an action of a finite group $\Gamma$.
Suppose that $\rkPic(X)^{\Gamma}=2$ and $K_X^2=3$.
Then~$X$ is a cubic surface containing a $\Gamma$-invariant line.
\end{proposition}

\begin{proof}
Since~{$\rkPic(X)^{\Gamma}=2$},
the group $\Pic(X)^{\Gamma}$ is generated by $-K_X$ and~$F$,
where $F$ is the class of a fiber of~$\phi$.
In this basis one has $K_X^2 = 3$, $K_X \cdot F = -2$, and $F^2 = 0$.

Observe that $X$ is geometrically rational. In particular, for any positive integer $m$ one has $h^0(X,\mathcal{O}_X(mK_X))=0$.
By Serre duality, this gives
$$
h^2(X,\mathcal{O}_X(-mK_X))=h^0(X,\mathcal{O}_X((m+1)K_X))=0.
$$
Thus, by Riemann--Roch theorem one has
$$
h^0(X, \mathcal{O}_X(-mK_X)) \geqslant
\chi(X, \mathcal{O}_X(-mK_X))
=\frac{m(m+1)}{2}K_X^2 + 1.
$$
Therefore, the divisor $-K_X$ is big.

Assume that the anticanonical class $-K_X$ is not nef. Then there exists a $\Gamma$-invariant curve $C$ such that $-K_X \cdot C < 0$, and
$$
|-K_X| = C + |M|,
$$
where $|M|$ is a moveable linear system of dimension
$$
\dim |-K_X| = h^0(X, \mathcal{O}_X(-K_X)) - 1
\geqslant K_X^2 = 3.
$$

One has
$$
C \sim -aK_X - bF
$$
with $0 < 3a < 2b$, since $C\cdot F\ge 0$ and $-K_X \cdot C < 0$. Thus
$$
M \sim (a-1)K_X + bF.
$$
Hence
$$
M^2 = 3(a-1)^2 -2(a-1)b = (a - 1)(3a - 3 - 2b).
$$
Since $3a < 2b$, this number can be non-negative only if $a = 1$, and the linear system~\mbox{$|M| \sim |bF|$} has dimension at least $3$ only if $b \geqslant 3$.
One has $C \sim -K_X - bF$ and~\mbox{$C \cdot F = 2$}, therefore there are three possibilities for $C_{\bar{\KK}}$ on $X_{\bar{\KK}}$: the curve $C_{\bar{\KK}}$ is either irreducible and generically reduced, or consists of two distinct sections of $\phi_{\bar{\KK}}$, or is a section of $\phi_{\bar{\KK}}$ with multiplicity~$2$ (in this case $C$ is not geometrically reduced).
The latter case is impossible since
$$
C^2 = K_X^2 + 2bK_X \cdot F = 3 - 4b
$$
is not divisible by $4$.

In the former two cases, observe that $C_{\bar{\KK}}$ is a Cartier divisor on a smooth surface, and each irreducible component of $C_{\bar{\KK}}$ is generically reduced: this implies that
$C_{\bar{\KK}}$ is reduced.
Since $C_{\bar{\KK}}$ consists of one or two irreducible components, one has~\mbox{$\mathrm{p}_a(C_{\bar{\KK}}) \geqslant -1$}.
On the other hand, by adjunction (see for instance~\mbox{\cite[Corollary 2.9]{Tan18}}) we compute
$$
2\mathrm{p}_a(C) - 2 = C \cdot (C + K_X) = -2b.
$$
Thus
$$
\mathrm{p}_a(C_{\bar{\KK}}) = \mathrm{p}_a(C) = 1 - b \leqslant -2,
$$
which gives a contradiction.

Now assume that the anticanonical class $-K_X$ is nef but not ample.
Since $K_X^2>0$, we conclude from Nakai--Moishezon criterion
that there exists a $\Gamma$-invariant $\Gamma$-irreducible reduced curve~$C$ such that $-K_X \cdot C = 0$. One has $C^2 < 0$ by Hodge index theorem since $-K_X$ is nef and big.
The curve $C$ has class $-aK_X - bF$. We see that $3a = 2b$, because~\mbox{$-K_X \cdot C = 0$}.
By adjunction we have
$$
C^2 = C^2 +  C \cdot K_X = C \cdot (C + K_X) = \operatorname{deg} \omega_C.
$$
Thus $h^0 (C, \omega_C) = 0$, and by Serre duality we get $h^1 (C, \mathcal{O}_C) = 0$. Hence
$$
C^2 = \operatorname{deg} \omega_C = 2(h^1 (C, \mathcal{O}_C) - h^0 (C, \mathcal{O}_C)) = -2h^0 (C, \mathcal{O}_C).
$$
Moreover, the map
$$
H^0 (C, \mathcal{O}_C) \to H^0 (C \cap F, \mathcal{O}_{C \cap F})
$$
is injective,
since $H^0 (C, \mathcal{O}_C)$ is a direct sum of finite extensions of $\KK$ transitively permuted by $\Gamma$. We observe that
$$
4a = 2C \cdot F = 2h^0 (C \cap F, \mathcal{O}_{C \cap F}) \geqslant 2h^0 (C, \mathcal{O}_C) = -C^2.
$$
On the other hand, one has
$$
C^2 = a^2K_X^2 + 2abK_X \cdot F + b^2F^2 = 3a^2-4ab= -3a^2.
$$
Therefore $3a^2 \leqslant 4a$, so we have either $a = 0$, or $a = 1$. Both of these cases are impossible, because~\mbox{$3a = 2b$}.

Therefore, we see that the anticanonical class $-K_X$ is ample, so that $X$ is a del Pezzo
surface of degree~\mbox{$K_X^2=3$}. Furthermore, by Lemma~\ref{lemma:unique-bisection} there is a line $E$ on $X$ such that
$$
E \sim -K_X - F
$$
in $\Pic(X)$. Therefore $E$ is $\Gamma$-invariant.
\end{proof}

\begin{corollary}
\label{corollary:Iskovskikh-elementary-transformation-equivariant}
Let $X$ be a smooth cubic surface over
a field $\KK$ with an action of a finite group $\Gamma$
such that~{$\rkPic(X)^\Gamma=2$} and $X$ contains a $\Gamma$-invariant line.
Let~\mbox{$\phi\colon X\to\PP^1$} be the $\Gamma$-equivariant conic bundle defined by this line,
and let~\mbox{$\theta\colon X\dashrightarrow X'$}
be a $\Gamma$-equivariant transformation of $\phi$.
Then $X'$ is a cubic surface such that~{$\rkPic(X')^\Gamma=2$} and~$X'$ contains a $\Gamma$-equivariant line.
In particular, $X'$ can be obtained by a blow up of a del Pezzo surface~$S'$ of degree $4$ with an action of $\Gamma$ at a
$\Gamma$-invariant $\KK$-point.
\end{corollary}

\begin{proof}
An equivariant transformation of a conic bundle
does not change the invariant Picard rank and the canonical degree, so that
the required assertion follows from Proposition~\ref{proposition:Iskovskikh-elementary-transformation-equivariant}.
\end{proof}

A smooth geometrically irreducible projective surface $Z$ with an action of a finite group~$\Gamma$ and a $\Gamma$-equivariant
morphism~\mbox{$\phi\colon Z\to B$}
is called a \emph{$\Gamma$-Mori fiber space} if~\mbox{$\dim Z>\dim B$}, the anticanonical divisor of $Z$ is $\phi$-ample, one
has
$$
\rkPic(Z)^{\Gamma}-\rkPic(B)^{\Gamma}=1,
$$
and~\mbox{$\phi_*\mathcal{O}_Z=\mathcal{O}_B$}. Thus, any $\Gamma$-Mori fiber space is either
a del Pezzo surface of $\Gamma$-invariant Picard rank~$1$,
or a relatively $\Gamma$-minimal conic bundle (i.e. a conic bundle of relative $\Gamma$-invariant Picard rank~$1$).

The next result is an equivariant version of Theorem~\ref{theorem:Iskovskikh-dP4-models}.

\begin{proposition}
\label{proposition:Iskovskikh-dP4-models-equivariant}
Let $S$ be a del Pezzo surface of degree~$4$ over
a field $\KK$ with an action of a finite group $\Gamma$
such that $\rkPic(S)^{\Gamma}=1$.
Then all the $\Gamma$-Mori fiber spaces $\Gamma$-equivariantly birational to $S$ are either del Pezzo
surfaces of $\Gamma$-invariant Picard rank~$1$ and degree~$4$, or smooth cubic surfaces with a structure
of a relatively $\Gamma$-minimal conic bundle over~$\PP^1$.
Furthermore, all $\Gamma$-equivariant birational maps from $S$ to $\Gamma$-Mori fiber spaces are
obtained by a composition of maps of the following types:
\begin{itemize}
\item $\Gamma$-equivariant automorphisms;

\item birational Geiser involutions of del Pezzo surfaces of degree $4$ defined \mbox{by $\Gamma$-invariant} points of degree $2$
or a $\Gamma$-invariant pairs of $\KK$-points in Sarkisov general position;

\item birational Bertini involutions of del Pezzo surfaces of degree $4$ defined \mbox{by $\Gamma$-invariant}
points of degree $3$ or a $\Gamma$-invariant triples of $\KK$-points in Sarkisov general position;

\item blow ups of $\Gamma$-invariant $\KK$-points in Sarkisov general position
on del Pezzo surfaces of degree $4$ with $\Gamma$-invariant Picard rank~$1$,
and their inverses;

\item $\Gamma$-equivariant transformations between conic bundles on smooth cubic surfaces of
Picard rank~$2$.
\end{itemize}
\end{proposition}

\begin{proof}[Sketch of a proof]
Every $\Gamma$-equivariant birational map from $S$ to another $\Gamma$-Mori fiber space
is a composition of $\Gamma$-Sarkisov links and automorphisms.
The latter can be classified similarly to the non-equivariant case, see Proposition~\ref{proposition:links-starting-from-dP4-equivariant}.
The remaining part of the argument goes as in the proof of Theorem~\ref{theorem:Iskovskikh-dP4-models}.
\end{proof}

To proceed, we need to analyze some of the birational maps listed in Proposition~\ref{proposition:Iskovskikh-dP4-models-equivariant}
in more detail.
Namely, in certain cases we show that the two surfaces involved in such a map, which we already know to
be equivariantly birational and (possibly non-equivariantly) isomorphic to each other, are in fact
equivariantly isomorphic.

\begin{lemma}\label{lemma:transformation-via-cubic-equivariant}
Let $S$ and $S'$ be del Pezzo surfaces of degree $4$ over
a field $\KK$
with an action of a finite group $\Gamma$
such that~\mbox{$\rkPic(S)^\Gamma=\rkPic(S')^\Gamma=1$}.
Let $\chi \colon S \dashrightarrow S'$ be a $\Gamma$-equivariant birational map.
Suppose that
$$
\chi = f^{-1}\theta f',
$$
where $f \colon X \rightarrow S$ and $f' \colon X' \rightarrow S'$ are blow ups of $\Gamma$-invariant $\KK$-points on del Pezzo surfaces of degree $4$ and $\theta \colon X \dashrightarrow X'$ is a $\Gamma$-equivariant transformation between conic bundles on smooth cubic surfaces.
Then there exists a $\Gamma$-equivariant isomorphism between~$S$ and~$S'$.
\end{lemma}

\begin{proof}
We have the following $\Gamma$-equivariant diagram
\begin{equation*}
\xymatrix{
&X\ar@{->}[d]_{\phi}\ar@{-->}[rr]^{\theta}\ar@{->}[ld]_{f}&& X'\ar@{->}[d]^{\phi'}\ar@{->}[rd]^{f'}&\\
S&\PP^1\ar@{->}[rr]^{\sim}&&\PP^1& S'
}
\end{equation*}

Consider any marking $L_1,\ldots,L_5$ of $S$. Applying Construction~\ref{construction:markings-1-to-1-S-to-X}, one obtains a marking
$$
E_1\cup F_1,\ldots,E_5\cup F_5
$$
of $\phi$.
Since the degenerate fibers (and their irreducible components)
of~$\phi$ and~$\phi'$ are naturally identified with each other,
this provides a marking of $\phi'$. Finally, applying Construction~\ref{construction:markings-1-to-1-X-to-S}, we obtain a marking $L_1',\ldots,L_5'$ of $S'$. Each of these markings gives an action of $\WD$ on $\Pic(S_{\KK^{sep}})$, $\Pic(X_{\KK^{sep}})$, $\Pic(X'_{\KK^{sep}})$, and $\Pic(S'_{\KK^{sep}})$, respectively.

The markings on $S$ and $S'$ induce the actions of $\WD$ on $\Lambda(S)$ and $\Lambda(S')$.
According to Proposition~\ref{proposition:Lambda-Lambda-prime-WD-equivariant} there exists a $\WD$-equivariant isomorphism
$$
\nu\colon\Lambda(S)\to \Lambda(S').
$$
Moreover, by Corollary \ref{corollary:dP4-isomorphic}(iv) there exists an isomorphism $\psi \colon S \stackrel{\sim}\longrightarrow S'$ such that the induced map $\tilde{\psi}\colon\Lambda(S)\to \Lambda(S')$ coincides with $\nu$.

The action of the group $\Gamma$ on $\Pic(S_{\KK^{sep}})$, $\Pic(X_{\KK^{sep}})$, $\Pic(X'_{\KK^{sep}})$, and $\Pic(S'_{\KK^{sep}})$ gives homomorphisms $\pi_S$, $\pi_X$, $\pi_{X'}$, and $\pi_{S'}$ from $\Gamma$ to $\WD$. All these homomorphisms coincide
by Lemmas~\ref{lemma:markings-wd-equivariant} and~\ref{lemma:markings-wd-equivariant-cb}.
Therefore the $\WD$-equivariant isomorphism~$\tilde{\psi}_{\KK^{sep}} = \nu_{\KK^{sep}}$ is $\Gamma$-equivariant. Thus for any line $L \in \Lambda(S)_{\KK^{sep}}$ and $g \in \Gamma$ one has
$$
g \tilde{\psi}_{\KK^{sep}}(L) = \tilde{\psi}_{\KK^{sep}} g(L).
$$
Hence the automorphism~\mbox{$(\psi_{\KK^{sep}} g)^{-1} (g \psi_{\KK^{sep}})$} preserves each line on $S_{\KK^{sep}}$. However, we know that the action of the automorphism group of $S_{\KK^{sep}}$
on the lines is faithful. Thus~\mbox{$(\psi g)^{-1} (g\psi)$} is the trivial automorphism of $S$. This means that the morphisms~\mbox{$g\psi \colon S \rightarrow S'$} and~\mbox{$\psi g \colon S \rightarrow S'$} coincide, so that $\psi$ is \mbox{a $\Gamma$-equivariant} isomorphism between $S$ and $S'$.
\end{proof}

Now one can deduce an equivariant analog of Corollary~\ref{corollary:Iskovskikh-dP4-models}.

\begin{corollary}
\label{corollary:Iskovskikh-dP4-models-equivariant}
Let $S$ be a del Pezzo surface of degree $4$ over
a field $\KK$
with an action of a finite group $\Gamma$
such that~\mbox{$\rkPic(S)^\Gamma=1$}.
Then all the $\Gamma$-Mori fiber spaces $\Gamma$-equivariantly birational to $S$ are~$S$ itself
and smooth cubic surfaces obtained by a blow up
of a $\Gamma$-invariant $\KK$-point on $S$.
\end{corollary}

\begin{proof}
Let $\chi \colon S \dashrightarrow S'$ be a $\Gamma$-equivariant birational map, where $S'$ is a del Pezzo surface of degree $4$
and $\Gamma$-invariant Picard rank~$1$.
By Proposition~\ref{proposition:Iskovskikh-dP4-models-equivariant},
the map $\chi$ can be decomposed into a sequence of birational maps
$$
S = S_0 \stackrel{\chi_0}\dashrightarrow S_1 \stackrel{\chi_1}\dashrightarrow \ldots \stackrel{\chi_{n-1}}\dashrightarrow S_n = S',
$$
where each $\chi_i$ is a $\Gamma$-equivariant birational map of one of the types listed in Proposition~\ref{proposition:Iskovskikh-dP4-models-equivariant}.
Using Lemmas~\ref{lemma:GB-involution-equivariant} and~\ref{lemma:transformation-via-cubic-equivariant}, we show that the surfaces $S$ and $S'$ are $\Gamma$-equivariantly isomorphic,
similarly to how it was done in the proof of Corollary~\ref{corollary:Iskovskikh-dP4-models}.

If $X'$ is a smooth cubic surface of $\Gamma$-invariant Picard rank~$2$ which is $\Gamma$-equivariantly birational to~$S$,
then by Proposition~\ref{proposition:Iskovskikh-elementary-transformation-equivariant} there is a $\Gamma$-invariant line $E$ on $X'$. One can blow down this line and obtain a del Pezzo surface $S'$ of degree $4$. But we
have already proved that $S'$ is $\Gamma$-isomorphic to $S$. Therefore, $X'$ is isomorphic to a blow up of a $\Gamma$-invariant $\KK$-point on $S$.
It remains to recall from Proposition~\ref{proposition:Iskovskikh-dP4-models-equivariant} that there are no other types of $\Gamma$-Mori fiber spaces
which are $\Gamma$-equivariantly birational to~$S$.
\end{proof}

Finally, we prove the main result of this section.

\begin{proof}[Proof of Theorem~\ref{theorem:dP4-unique-equivariant}]
Assertion~(ii) is given by Corollary~\ref{corollary:Iskovskikh-dP4-models-equivariant}.
Assertion~(i) also follows from Corollary~\ref{corollary:Iskovskikh-dP4-models-equivariant}
because if there are no $\Gamma$-invariant $\KK$-points on $S$, then no cubic surface
can be obtained from $S$ by blowing up such a point.
\end{proof}

\appendix

\section{Geiser and Bertini involutions}
\label{appendix:involutions}

In this section we recall some properties of (birational) Geiser and Bertini involutions on regular del Pezzo surfaces,
and apply them to prove Theorem~\ref{theorem:lowdegree} and sketch a proof of Theorem~\ref{theorem:G-lowdegree}.

Let us start with (a variation of) a well known definition.

\begin{definition}[{cf. \cite[Definition 4.34]{BFSZ}}]
\label{definition:Sarkisov-general}
Let $S$ be a regular del Pezzo surface, and let $P\subset S$ be a collection of distinct points.
We say that $P$ is \emph{in Sarkisov general position} if the blow up of $S$ at $P$ is again a regular del Pezzo surface.
\end{definition}

\begin{proposition}
\label{proposition:GeiBerBir}
Let $S$ be a geometrically integral regular del Pezzo surface of degree~$d$ over a field~$\KK$, and let~\mbox{$P \subset S$} be a collection of distinct points such that the sum of their degrees
equals~\mbox{$d-2$}
(respectively,~\mbox{$d-1$}). Suppose that $P$ is in Sarkisov general position.
Then there exists a birational involution
$$
\chi \colon S \dashrightarrow S
$$
such that the indeterminacy locus of $\chi$ coincides with $P$ and $\chi$ regularizes on $Y$,
where $Y$ is the blow up of $S$ at $P$.
\end{proposition}

\begin{proof}
Let $\pi \colon Y \to S$ be the blow up of $S$ at $P$ and $E$ be the exceptional divisor of~$\pi$. Then~$Y$ is a del Pezzo surface of degree $2$ (respectively, $1$), because $P$ is in Sarkisov general position. Moreover, one has~\mbox{$\rkPic(Y) \geqslant 2$}, and therefore there exists the Geiser (respectively, Bertini) involution $\sigma$ on $Y$ by~\mbox{\cite[Proposition 4.18]{BFSZ}}. Thus
the birational map
$$
\chi = \pi \sigma \pi^{-1}
$$
is the required birational involution on $S$.
\end{proof}

\begin{definition}\label{definition:birational-GB}
The birational involution $\chi \colon S \dashrightarrow S$ constructed in Proposition~\ref{proposition:GeiBerBir}
is called a \emph{birational Geiser involution} (respectively, a \emph{birational Bertini involution}).
\end{definition}

\begin{corollary}
\label{corollary:GeiBerIso}
Let $S$ be a geometrically integral regular del Pezzo surface of degree~$d$ over a field~$\KK$ such that~\mbox{$\rkPic(S)=1$}, and $p\in S$ be a point of degree $d-2$ (respectively,~\mbox{$d-1$}).
Suppose that there is a Sarkisov link $\chi\colon S \dashrightarrow S'$ starting with the blow up of $P$. Then~$\chi$
is a birational Geiser (respectively, Bertini) involution. In particular, one has $S' \cong S$.
\end{corollary}

\begin{proof}
Let $\pi \colon Y \to S$ be the blow up of $S$ at $p$.
Then $Y$ is a geometrically integral regular del Pezzo surface (and so the point $p$ is
in Sarkisov general position).
Denote by $\sigma$ the Geiser (respectively, Bertini) involution on $Y$.
Since $\rkPic(Y) = 2$, we know from~\mbox{\cite[Proposition 4.21]{BFSZ}} that $\sigma$ interchanges the two extremal rays of the nef cone of $Y$.
Therefore, for the second extremal contraction $\pi'\colon Y \to S'$ one has~\mbox{$\pi' = \pi \sigma$}.
Thus, it follows from Proposition~\ref{proposition:GeiBerBir} that $\chi$ is
a birational Geiser (respectively, Bertini) involution, and $S' \cong S$.
\end{proof}

Now we prove Theorem~\ref{theorem:lowdegree}.

\begin{proof}[Proof of Theorem~\ref{theorem:lowdegree}]
By \cite[Theorem A]{BFSZ}, any birational map between regular Mori fiber spaces of dimension $2$ is a composition of Sarkisov links and automorphisms of Mori fiber spaces.

From~\mbox{\cite[Propositions 5.1 and 5.5]{BFSZ}} one can see that for a geometrically integral regular del Pezzo surface~$S$ of Picard rank~$1$ and degree~$d$,
any Sarkisov link $\chi\colon S\dasharrow S'$ starts with the blow up of a point~$p$ of degree less than $d$. Moreover, if $\operatorname{deg} (p) = d - 1$ or~\mbox{$\operatorname{deg} (p) = d - 2$}, then one has $S' \cong S$ by Corollary~\ref{corollary:GeiBerIso}. Thus $S$ is birationally rigid for~\mbox{$d \leqslant 3$}.

Observe that if $d = 3$ and $S(\KK)=\varnothing$, then there are no points of degree $2$ on $S$. Indeed, for a line $L$ passing through a point of degree $2$ on a cubic surface $S$ in $\PP^3$,
one has~\mbox{$\operatorname{deg} (S \cap L) = 3$} (note that the line $L$ is not contained in $S$ since $\rkPic(S)=1$). Therefore, the intersection~\mbox{$S \cap L$} contains a $\KK$-point.
Thus there are no Sarkisov links starting from $S$ if $d\le 3$ and $S(\KK)=\varnothing$, or if $d = 1$. This means that $S$ is birationally super-rigid in these cases.
\end{proof}

Now we study $\Gamma$-equivariant birational Geiser and Bertini involutions.

\begin{lemma}\label{lemma:GB-involution-equivariant}
Let $S$ be a geometrically integral regular del Pezzo surface of degree~$d$ over
a field $\KK$
with two actions of a finite group $\Gamma$ given by homomorphisms
$\rho\colon \Gamma\to \Aut(S)$ and $\rho'\colon \Gamma\to \Aut(S)$.
Suppose that
$$
\rkPic(S)^{\rho(\Gamma)}=\rkPic(S)^{\rho'(\Gamma)}=1.
$$
Let $\chi \colon S \dashrightarrow S$
be a birational Geiser or Bertini involution which is equivariant with respect to these two actions of $\Gamma$ on the two copies of $S$;
in other words, one has
$$
\rho'(g)\chi=\chi\rho(g)
$$
for any $g\in \Gamma$.
Then $\rho=\rho'$.
In particular, the identity automorphism of $S$ is equivariant with respect to the two actions of $\Gamma$ on the two copies of $S$.
\end{lemma}

\begin{proof}
By construction of a birational Geiser (respectively, Bertini) involution (see~Proposition~\ref{proposition:GeiBerBir}), one has
$$
\chi = \pi\sigma\pi^{-1},
$$
where $\pi \colon Y \rightarrow S$ is the blowup of a $\rho(\Gamma)$-invariant collection of distinct points such that the sum of their degrees equals to~\mbox{$d-2$} (respectively,~\mbox{$d-1$}),
the surface $Y$ is a geometrically integral regular del Pezzo
surface of degree~$2$ (respectively,~$1$), and $\sigma$ is a Geiser (respectively, Bertini) involution on~$Y$.
Denote by $\rho_Y$ the homomorphism of $\Gamma$ to the group $\Aut(Y)$
which defines the action of~$\Gamma$ on $Y$. Furthermore,
the diagram
$$
\xymatrix{
& Y\ar@{->}[ld]_\pi\ar@{->}[rd]^{\pi\sigma} &\\
S\ar@{-->}[rr]^\chi && S
}
$$
is $\Gamma$-equivariant. This means that the maps $\pi$ and $\pi\sigma$ are equivariant, or, to be more precise, that
\begin{equation}\label{eq:pi-equivariant}
\rho(g)\pi=\pi\rho_Y(g), \quad \rho'(g)\pi\sigma=\pi\sigma\rho_Y(g)
\end{equation}
for any element $g\in\Gamma$.

Let us show that for any $\alpha \in \Aut(Y_{\KK^{sep}})$
one has
\begin{equation}\label{eq:sigma-equivariant}
\alpha\sigma = \sigma \alpha.
\end{equation}
Denote by $\Sigma\cong\ZZ/2\ZZ$ and $\Sigma'\cong\ZZ/2\ZZ$ the subgroups of $\Aut(Y_{\KK^{sep}})$ generated by the involutions
$\sigma$ and $\sigma' = \alpha\sigma\alpha^{-1}$, respectively.
Then the quotient~\mbox{$Y_{\KK^{sep}} / \Sigma$}
is isomorphic to~$\PP^2_{\KK^{sep}}$ (respectively, to the weighted projective plane~\mbox{$\PP_{\KK^{sep}}(1,1,2)$}), and the map
$$
Y_{\KK^{sep}} \rightarrow Y_{\KK^{sep}} / \Sigma
$$
is given by the complete linear system $\mathcal{L} = |-K_{Y_{\KK^{sep}}}|$ (respectively, by~\mbox{$\mathcal{L} = |-2K_{Y_{\KK^{sep}}}|$}), see~\cite[Proposition 4.18]{BFSZ}. Therefore, the quotient map
$$
Y_{\KK^{sep}} \rightarrow Y_{\KK^{sep}} / \Sigma'
$$
is given by the linear system $\alpha(\mathcal{L})$. But the anticanonical linear system is invariant under the action of any automorphism. Thus one has $\mathcal{L} = \alpha(\mathcal{L})$, and the map $Y_{\KK^{sep}} \rightarrow Y_{\KK^{sep}} / \Sigma'$ is given by $|-K_{Y_{\KK^{sep}}}|$ (respectively, by~$|-2K_{Y_{\KK^{sep}}}|$). Hence the groups $\Sigma$ and $\Sigma'$ coincide, and one has $\sigma = \alpha\sigma\alpha^{-1}$.
In other words, equality~\eqref{eq:sigma-equivariant} holds.

Now, using~\eqref{eq:pi-equivariant} and~\eqref{eq:sigma-equivariant}, for any $g \in \Gamma$ we compute
$$
\rho'(g) \pi\sigma=
\pi\sigma \rho_Y(g)=
\pi\rho_Y(g)\sigma=
\rho(g)\pi\sigma.
$$
This implies that $\rho(g)=\rho'(g)$ and completes the proof of the lemma.
\end{proof}

Now we can give a sketch of a proof of Theorem~\ref{theorem:G-lowdegree}.
To turn it into a rigorous proof, one needs to
establish the existence of decomposition of a $\Gamma$-equivariant birational map into $\Gamma$-Sarkisov links, and classify such links
for regular surfaces over arbitrary fields using the same arguments as in~\cite{BFSZ}.

\begin{proof}[Sketch of a proof of Theorem~\ref{theorem:G-lowdegree}]
Every $\Gamma$-equivariant birational map from $S$ to another regular $\Gamma$-Mori fiber space
is a composition of $\Gamma$-Sarkisov links and automorphisms.

For a geometrically integral regular del Pezzo surface $S$ of $\Gamma$-invariant Picard rank~$1$ and degree~$d \leqslant 3$ the only possible $\Gamma$-Sarkisov links $S \dashrightarrow S'$ are birational Geiser involutions for~\mbox{$d = 3$}, and birational Bertini involutions for $d = 3$ and $d = 2$. Therefore, one has~\mbox{$S' \cong S$} by Proposition~\ref{proposition:GeiBerBir}, and moreover there exists a $\Gamma$-equivariant isomorphism of~$S$ and~$S'$ by Lemma~\ref{lemma:GB-involution-equivariant}. Thus, $S$ is $\Gamma$-birationally rigid.

Finally, observe that for $d \leqslant 3$ there are no $\Gamma$-Sarkisov links starting from $S$ if there are no $\Gamma$-invariant $\KK$-points on $S$, or if $d = 1$. Hence
$S$ is birationally super-rigid in these cases.
\end{proof}

\section{Sarkisov links}
\label{appendix:pointless-4}

In this section we recall the results of~\cite{BFSZ}
concerning Sarkisov links from geometrically integral regular del Pezzo surfaces of degree~$4$,
and apply them to prove Theorem~\ref{theorem:pointless-dP4} and sketch a proof of Theorem~\ref{theorem:G-pointless-dP4}.

\begin{proposition}
\label{proposition:links-starting-from-dP4}
Let $S$ be a geometrically integral regular del Pezzo surface of degree~$4$ over
a  field $\KK$
such that~\mbox{$\rkPic(S)=1$}.
Then all Sarkisov links starting from $S$ are of one of the following types:
\begin{itemize}
\item link of type $\mathrm{II}$ which is a birational Geiser involution $S\dasharrow S$ defined by a point of degree~$2$ in Sarkisov general
position;

\item link of type $\mathrm{II}$ which is a birational Bertini involution defined by a point of degree~$3$ in Sarkisov general
position;

\item link of type $\mathrm{I}$ which is a blow up of a $\KK$-point on $S$
in Sarkisov general position resulting in a regular cubic surface with a structure of a conic bundle.
\end{itemize}
\end{proposition}

\begin{proof}
From~\mbox{\cite[Propositions 5.1 and 5.5]{BFSZ}}
one can see that there are three kinds of Sarkisov links from $S$, namely,
links of type $\mathrm{I}$ starting with the blow up of a $\KK$-point~$p$,
and links of type $\mathrm{II}$ starting with a blow up of a point~$p$ of degree $2$ or~$3$.
Furthermore, if~$p$ is a $\KK$-point, and $\pi\colon X \rightarrow S$ is the blow up of $S$ at $p$, then $X$ is a regular cubic surface admitting a structure of a conic bundle.
On the other hand, if the point $p$ has degree~$2$ (respectively, degree~$3$), then the corresponding Sarkisov link $S \dashrightarrow S'$
is a birational Geiser (respectively, Bertini) involution, and one has $S' \cong S$, see Corollary~\ref{corollary:GeiBerIso}.
\end{proof}

Using Proposition~\ref{proposition:links-starting-from-dP4},
we prove Theorem~\ref{theorem:pointless-dP4}.

\begin{proof}[Proof of Theorem~\ref{theorem:pointless-dP4}]
By \cite[Theorem A]{BFSZ}, every birational map from $S$ to another regular Mori fiber space is a composition of Sarkisov links and automorphisms of Mori fiber spaces.
Since $S$ has no $\KK$-points, we see from Proposition~\ref{proposition:links-starting-from-dP4} that any Sarkisov link
starting from $S$ is a birational involution of $S$. Hence $S$ is birationally rigid.
Moreover, if $S$ has no points of degree $2$ and $3$, then there are no Sarkisov links starting from $S$ at all,
which means that $S$ is birationally super-rigid.
\end{proof}

The next assertion is an equivariant version of Proposition~\ref{proposition:links-starting-from-dP4} which is well known to experts, at least over perfect fields.

\begin{proposition}
\label{proposition:links-starting-from-dP4-equivariant}
Let $S$ be a geometrically integral regular del Pezzo surface of degree~$4$ over
a field $\KK$ with an action of a finite group $\Gamma$
such that $\rkPic(S)^{\Gamma}=1$.
Then all $\Gamma$-Sarkisov links starting from $S$ are of one of the following types:
\begin{itemize}
\item link of type $\mathrm{II}$ which is a birational Geiser involution $S\dasharrow S$ defined by \mbox{a $\Gamma$-invariant} point of degree $2$
or a $\Gamma$-invariant pair of $\KK$-points in Sarkisov general position;

\item link of type $\mathrm{II}$ which is a birational Bertini involution defined by a $\Gamma$-invariant
point of degree $3$ or a $\Gamma$-invariant triple of $\KK$-points in Sarkisov general position;

\item link of type $\mathrm{I}$ which is a blow up of a $\Gamma$-invariant $\KK$-point on $S$
in Sarkisov general position resulting in a regular cubic surface with a structure of a $\Gamma$-equivariant conic bundle.
\end{itemize}
\end{proposition}

Using Proposition~\ref{proposition:links-starting-from-dP4-equivariant}, we sketch a proof of Theorem~\ref{theorem:G-pointless-dP4}.

\begin{proof}[Sketch of a proof of Theorem~\ref{theorem:G-pointless-dP4}]
Every $\Gamma$-equivariant birational map from $S$ to another regular $\Gamma$-Mori fiber space
is a composition of $\Gamma$-Sarkisov links and automorphisms.
Since~$S$ has no $\Gamma$-invariant $\KK$-points, we see from Proposition~\ref{proposition:links-starting-from-dP4-equivariant} and Lemma~\ref{lemma:GB-involution-equivariant} that any $\Gamma$-Sarkisov link starting from $S$ is a $\Gamma$-equivariant birational involution of $S$.
Hence~$S$ is $\Gamma$-birationally rigid.
Moreover, if $S$ has no $\Gamma$-invariant points of degree $2$ and $3$, and no $\Gamma$-invariant pairs and triples of $\KK$-points,
then there are no $\Gamma$-Sarkisov links starting from~$S$ at all,
which means that $S$ is $\Gamma$-birationally super-rigid.
\end{proof}

\begin{remark}
One can get rid of the assumpition on the absence of
points of degree $3$ in Theorem~\ref{theorem:pointless-dP4},
and similarly of the assumption on the absence of
$\Gamma$-invariant points of degree $3$ and $\Gamma$-invariant triples of $\KK$-points
in Theorem~\ref{theorem:G-pointless-dP4}. Indeed, if a geometrically integral regular del Pezzo
surface of degree $4$ over a field $\KK$ has a point of degree $3$, then it also has a $\KK$-point, and
an analogous assertion holds in the equivariant setting. This can be proved using the arguments
similar to those in the proof of~\mbox{\cite[Lemma~2.4]{ShramovVikulova}}.
\end{remark}

\end{document}